\documentclass[a4paper]{article}
\usepackage[english]{babel}					
\usepackage[utf8]{inputenc}					
\usepackage{fullpage}						
\usepackage[T1]{fontenc}						
\usepackage{amsmath}							
\usepackage{amsthm}							
\usepackage{amsfonts}						
\usepackage{amssymb}							
\usepackage{bbm}								
\usepackage{stmaryrd}						
\usepackage{tikz}							
\usepackage{graphicx}						
\usepackage{array}							
\usepackage{relsize}							

\theoremstyle{plain}							
\newtheorem{thm}{Theorem}[section]			
\newtheorem{lem}[thm]{Lemma}					
\newtheorem{cor}[thm]{Corollary}	
\newtheorem{prop}[thm]{Property}
\newtheorem{propo}[thm]{Proposition}
\newtheorem{Def}[thm]{Definition}
\newtheorem{thmref}{Lemma}[section]			

\theoremstyle{definition}					
\newtheorem{rmq}[thm]{Remark}

\DeclareMathOperator*{\argmin}{argmin}		
\reversemarginpar							
\newcommand{\od}{\partial{}}					
\newcommand{\ud}{\mathrm{d}}					
\newcommand{\RR}[1]{\mathbb{R}^{#1}}			
\newcommand{\NN}{\mathbb{N}}					
\newcommand{\ind}[1]{\mathbbm{1}_{#1}}		
\newcommand{\ee}{\varepsilon}				
\newcommand{\ff}{\varphi}					
\newcommand{\HH}[1]{\mathcal{H}^{#1}}		

\newcommand{\Div}{\textup{div}}
\newcommand{\WW}{\overline{W}}
\newcommand{\Wp}{\overline{W}_p}
\newcommand{\OO}{\Omega}
\newcommand{\FF}{\mathcal{F}}

\title{The $L^1$ gradient flow of a generalized scale invariant Willmore energy for radially non increasing functions.}
\author{François DAYRENS}

\begin{document}

\maketitle

\begin{abstract}
We use the minimizing movement theory to study the gradient flow associated to a non-regular relaxation of a geometric functional derived from the Willmore energy. Thanks to the coarea formula, we can define a Willmore energy on regular functions of $L^1(\mathbb{R}^d)$. This functional is extended to every $L^1$ function by taking its lower semicontinuous envelope. We study the flow generated by this relaxed energy for radially non-increasing functions (functions with balls as superlevel sets). In the first part of the paper, we prove a coarea formula for the relaxed energy of such functions. Then we show that the flow consists of an erosion of the initial data. The erosion speed is given by a first order ordinary equation.
\end{abstract}

\noindent \textbf{MSC classification:} 49J45, 49J52, 46S30.

\noindent \textbf{Keywords:} minimizing movement, Willmore energy, radial functions, gradient flow.


\section{Introduction and settings.}

This paper is devoted to De Giorgi's minimizing movement solutions of a generalized scale invariant Willmore flow. These solutions are obtained as limits of discrete generalized Willmore flows built upon the minimization in $L^1(\RR{d})$ of a relaxation of the functional
\begin{equation} \label{intro min}
\FF(u)=\mathlarger\int_{\RR{d}} |\nabla u| \left| \Div \frac{\nabla u}{|\nabla u|} \right|^{d-1} \ud x
+ \frac{1}{2\tau}||u-u_n||^2_{L^1}
\end{equation}
where $u_n$ is updated at each discrete time step and $u_0\in L^1$ is an initial datum (see below for more details on minimizing movement solutions).

The first term in $\FF$ is a case of the generalized $p$-Willmore energy
\[ W_p(u)=\mathlarger\int_{\RR{d}} |\nabla u| \left| \Div \frac{\nabla u}{|\nabla u|} \right|^{p} \ud x \]
which arise in image processing and 2D or 3D shape completion \cite{Ambrosio2003} \cite{Bredies} \cite{Droske2004} \cite{Masnou2006}. 

Using the coarea formula (see Theorem \ref{coarea general}), one can decompose $W_p(u)$ on every superlevel sets:
\[ W_p(u) = \int_{-\infty}^{+\infty} \left( \int_{\od \{u\geqslant t\} } 
\left| \overrightarrow{H} \right|^p \ud \HH{d-1} \right) \ud t \]
with $\overrightarrow{H}$ the mean curvature vector on the boundary of the superlevel set $\{ u\geqslant t \}$ and $\HH{d-1}$ the $(d-1)$-Hausdorff measure.

This leads us to consider the $p$-Willmore energy of a hypersurface $M$ of $\RR{d}$:
\[ W_p(M)= \int_M |H|^p \ud A \]
where $H$ is its scalar mean curvature and $\ud A$ its area measure. As one can see, for $p=d-1$, $W_{d-1}(\lambda M)=W_{d-1}(M)$ for any $\lambda>0$ therefore $W_{d-1}$ is scale invariant. We refer to \cite{Anzellotti1990} to see how the direct method of calculus of variations can be used on that kind of energies. 

For $p=2$ and $d=3$, $W_2$ is the traditional conformal Willmore energy of a surface in $\RR{3}$ and its classical smooth flow is given by
\[ \frac{\ud}{\ud t} M_t = \Delta_{M_t}H_t+2H_t(H_t^2-K_t) \]
with $\Delta_{M_t}$ the Laplace-Beltrami operator on $M_t$, $H_t$ its mean curvature and $K_t$ its Gauss curvature. We refer to \cite{Simonett2001}, \cite{Kuwert2001,Kuwert2004} and \cite{Droske2004} for topics concerning this regular flow.

For $p=2$ and $d=2$, $W_2$ is the famous Bernoulli-Euler \emph{elastica}:
\[ \int \kappa^2 \ud s \]
where $\kappa$ is the curvature of a regular curve in the plane. The flow of a family of curves $(\gamma_t)$ is given by the equation
\[ \frac{\ud}{\ud t} \gamma_t = \left( \kappa_t''+\frac{1}{2}\kappa_t^3 \right) \vec{n}_t \]
with $\vec{n}_t$ the unit normal vector of $\gamma_t$. The elastica theory is a part of the solid materials mechanics initially developed by Euler and Bernoulli and devoted to the study of the deflection of thin beams. We refer to \cite{Sachkov2008} for a historical overview. Elastica has many applications in shape completion and inpainting, see \cite{Cao2011,Chan2002} for an account on these topics. One may also refer to \cite{Bredies} for a functional lifting approach on elastica approximation in the context of image processing. Finally, reader will find in \cite{Citti2006} links between elastica and amodal completion methods based on human visual perception.

As we can see from the previous formulas, the smooth Willmore flow is a geometric flow defined by a fourth order equation. A celebrated example of geometric flow is the mean curvature flow, which is defined by a second order equation derived from the area of a hypersurface and which received a lot of attention. For instance, in \cite{Huisken1990}, Huisken studied the regular flow and singularity appearance. To deal with singularities, the flow is extended in \cite{Brakke1978} to varifolds, a weak generalization of surfaces. The so called level set approach to mean curvature flow is introduced in \cite{Chen1991} and \cite{Evans1991,Evans1992a,Evans1992b,Evans1995}, see \cite{Giga2006} for a good overview on surface evolution using level set methods. Finally, \cite{Chambolle2004} and \cite{Luckhaus1995} use a minimizing movement approach to study this flow and the second paper even provides an algorithm to compute it.

Among many generalizations of the area of a surface, let us focus on the theory developed in the framework of $BV$ functions, i.e. functions with bounded variation. In this setting, the area of the boundary of a set $E$ coincides with the total variation of its characteristic functions:
\[ P(E)=TV(\ind{E}) \]
where $TV$ is the total variation (see \cite{Ambrosio2000,Evans1992,Ziemer1989}) and $P$ the perimeter of $E$ (or the area of $\od E$). The coarea formula for BV functions states that
\[ TV(u)=\int_{-\infty}^{+\infty} P(\{ u\geqslant t\} ) \ud t . \]
It is natural to wonder whether the gradient flow of $TV$ can be recovered from the individual gradient flows of the perimeter of each superlevel set. Authors in \cite{Bellettini2002} show that the total variation flow is different and, when starting with $u_0=\ind{E}$, it follows that
\[ u_t=(1-c_Et)_+\ind{E} \]
with $c_E$ a constant depending only on $E$. In contrast, for the individual flow of superlevel sets, it is prooved in \cite{Luckhaus1995} that the minimizing movement flow for $P$ corresponds to the mean curvature flow for regular surfaces. Besides, $TV$ is a first order functional, similar to our second order functional $\FF$, used in \cite{Rudin1992} for image denoising problems.

Returning to our fourth order flow in $L^1$ defined by the minimization of (\ref{intro min}), we will prove a very similar evolution for radially non increasing functions in the weaker setting of De Giorgi's minimizing movement theory \cite{Almgren1993,DeGiorgi1993,Luckhaus1995}. We start with a non negative, bounded, compactly supported and radially non increasing initial datum $u_0:\RR{d}\to\RR{}$. That means there exists a function $r:[0,a]\to\RR{+}$ such that
\[ \{ u_0\geqslant t\} = B(0,r(t)) .\]
We prove in this paper that the minimizing movement flow is an erosion of the initial datum. This erosion is described by an ordinary differential equation involving only the radius function $r$.

The plan of the paper is as follows: in section \ref{partie coarea}, we prove a coarea formula for the Willmore energy. This allows us to show in section \ref{partie mouvement} that a minimizer of (\ref{intro min}) is radially non increasing. We iterate to construct a minimizing sequence and then we describe the minimizing movement. The found flow is very close to the total variation flow in \cite{Bellettini2002}, i.e. an erosion of the initial datum. We prove in Theorem \ref{mvt min rad} that the minimizing movement $t\mapsto u_t$ exists and is described through a function $\lambda : [0,+\infty[ \to \RR{+}$ (see Figure \ref{intro}):
\[ u_t(x)=\min(u_0(x),\lambda(t)). \]

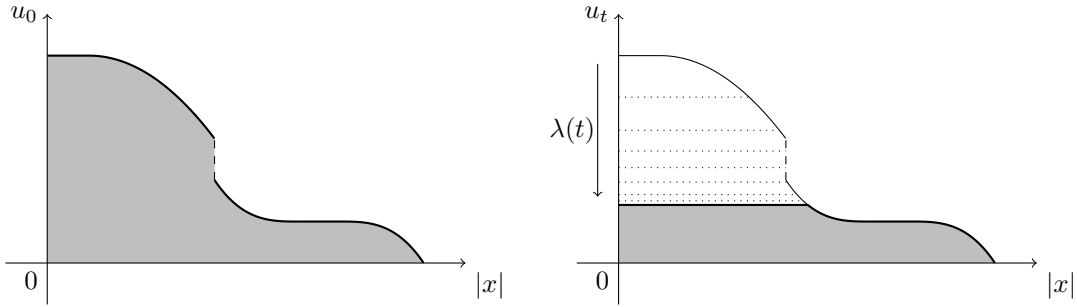
\begin{figure}[!ht]

\begin{center}
\begin{tikzpicture}[scale=0.55]
\fill [color=gray!50] (0,5) -- (1,5) --
plot [domain=1:4,samples=100] (\x,{ 5-(2/9)*(\x-1)^2 })
-- (4,2) -- (4,3) --
plot [domain=4:6,samples=100] (\x,{ 1-(1/8)*(\x-6)^3 }) 
-- (6,1) -- (7,1) --
plot [domain=7:9,samples=100] (\x,{ 1-(1/8)*(\x-7)^3 })
-- (9,0) -- (0,0) -- cycle ;
\draw (0,0) node[below left] {$0$} ;
\draw [->] (-1,0) -- (10,0) node[below right] {$|x|$} ;
\draw [->] (0,-1) -- (0,6) node[left] {$u_0$} ;
\draw [thick] (0,5) -- (1,5) 
plot [domain=1:4,samples=100] (\x,{ 5-(2/9)*(\x-1)^2 })
plot [domain=4:6,samples=100] (\x,{ 1-(1/8)*(\x-6)^3 }) 
(6,1) -- (7,1)
plot [domain=7:9,samples=100] (\x,{ 1-(1/8)*(\x-7)^3 }) ;
\draw [dashed] (4,2) -- (4,3) ;
\end{tikzpicture}
~
\begin{tikzpicture}[scale=0.55]
\fill [color=gray!50] (0,1.4) -- (4.52,1.4) --
plot [domain=4.52:6,samples=100] (\x,{ 1-(1/8)*(\x-6)^3 }) 
-- (6,1) -- (7,1) --
plot [domain=7:9,samples=100] (\x,{ 1-(1/8)*(\x-7)^3 })
-- (9,0) -- (0,0) -- cycle ;
\draw (0,0) node[below left] {$0$} ;
\draw [->] (-1,0) -- (10,0) node[below right] {$|x|$} ;
\draw [->] (0,-1) -- (0,6) node[left] {$u_t$} ;
\draw (0,5) -- (1,5) 
plot [domain=1:4,samples=100] (\x,{ 5-(2/9)*(\x-1)^2 })
plot [domain=4:4.52,samples=100] (\x,{ 1-(1/8)*(\x-6)^3 }) ;
\draw [thick] (0,1.4) -- (4.52,1.4)
plot [domain=4.52:6,samples=100] (\x,{ 1-(1/8)*(\x-6)^3 }) 
(6,1) -- (7,1)
plot [domain=7:9,samples=100] (\x,{ 1-(1/8)*(\x-7)^3 }) ;
\draw [dotted] (0,4) -- (3.12,4) ;
\draw [dotted] (0,3.2) -- (3.84,3.2) ;
\draw [dotted] (0,2.7) -- (4,2.7) ;
\draw [dotted] (0,2.3) -- (4,2.3) ;
\draw [dotted] (0,1.95) -- (4.03,1.95) ;
\draw [dotted] (0,1.65) -- (4.26,1.65) ;
\draw [dotted] (0,1.5) -- (4.4,1.5) ;
\draw [dashed] (4,2) -- (4,3) ;
\draw [->] (-0.5,4.8) -- (-0.5,1.6) ;
\draw (-1.1,3.2) node {$\lambda(t)$} ;
\end{tikzpicture}
\end{center}

\caption{Minimizing movement of a radially non-increasing function for the Willmore energy.}
\label{intro}

\end{figure}

\noindent Moreover, Theorem \ref{EDO} states that $\lambda$ is a solution of the ordinary differential equation
\begin{equation} \label{EDO intro}
\left\{ \begin{array}{rcl}
\lambda'(t) & = & \displaystyle -\frac{\omega_d}{\alpha_d^2} \frac{1}{r_0(\lambda(t))^{2d}} \\
\lambda(0) & = & \sup u_0 , \end{array} \right. \end{equation}
where $\frac{\omega_d}{\alpha_d^2}$ is simply a dimensional constant. ODE (\ref{EDO intro}) implies one question we cannot answer for now: what happens if the initial condition $\lambda(0)$ satisfies $r(\lambda(0))=0$? Implicitly, ODE (\ref{EDO intro}) indeed says the erosion starts with a non finite speed.

\subsection{Minimizing movements.}

We now recall the notion of minimizing movement. The idea is to define the gradient flow for a possibly non regular functional. Minimizing movements can be defined in Banach spaces but are easier to understand and to handle in Hilbert spaces, see \cite{Ambrosio1995,Ambrosio2005,DeGiorgi1993} for a good account on this theory. In this paper we will build a minimizing movement in a self contained way. Let us first recall the basics of this theory on Hilbert spaces.

Let $H$ be a Hilbert space and consider
\[ f : H \to \overline{\RR{}} . \]

If $f$ is regular (at least $C^1$) then its (minimizing) gradient flow is defined by the ordinary differential equation:
\begin{equation} \label{ODE min}
\left\{ \begin{array}{rcl}
u'(t) & = & -\nabla f(u(t)) \\
u(0) & = & x.
\end{array} \right.
\end{equation}
In order to extend this notion to a non regular functional, we can use the implicit Euler scheme for solving this ordinary differential equation: for a small time parameter $\tau > 0$:
\[ \frac{u_{n+1}-u_n}{\tau} = -\nabla f(u_{n+1}). \]
Write
\[ F(y)=f(y)+\frac{1}{2\tau}||y-u_n||^2 \]
and the equation becomes
\[ \nabla F(u_{n+1}) = 0. \]
Thus, we transform ODE (\ref{ODE min}) into the following minimization problem: given a starting point $x\in H$, find a sequence $(u_n)$ satisfying $u_0=x$ and
\[ u_{n+1} \in \argmin_{v\in H} \left( f(v)+\frac{1}{2\tau}||v-u_n||^2 \right) . \]
Moreover, it yields the same solutions for regular functions as shown by Theorem \ref{mvt regulier}.

\begin{Def}[Minimizing sequence]
Let $\tau > 0$ and $x_0 \in$ \textup{dom}($f$)$=\{ x \in H \ | \ -\infty < f(x) < +\infty \}$, the minimizing sequence $(x_\tau (n))$ of $f$ starting from $x_0$ is defined by induction: \\
$x_\tau (0)=x_0$ and $x_\tau (n+1)$ is a minimizer of
\[ y \mapsto f(y) + \frac{1}{2\tau}||y-x_\tau(n)||^2 . \]
\end{Def}

\noindent From that sequence, we can define a piecewise constant path
\[ u_\tau (t) = x_\tau ([t/\tau]) \]
with $[.]$ the integer part.

\begin{Def}[Minimizing movement]
We say that $u:[0,+\infty[ \to H$ is a minimizing movement for $f$ starting from $x_0$ if there exists a minimizing sequence such that $u$ is a uniform limit of $u_\tau$ when $\tau \to 0$ (up to a subsequence $\tau_i \to 0$).
\end{Def}

\begin{rmq}
If the starting point $x_0$ is a local finite minimizer of $f$, then $u(t)=x_0$ is the only minimizing movement starting from $x_0$. Indeed (within  a constant we can suppose that this minimum is positive) in a ball $B(x_0,\delta)$ we have $0<f(x_0) \leqslant f(x)$. Then we obtain
\[ f(x_0) \leqslant f(x) + \frac{1}{2\tau}||x-x_0||^2 \quad \textup{on } B(x_0,\delta) \]
with equality only when $x=x_0$. \\
For $\tau$ small enough $\tau < \frac{\delta^2}{2f(x_0)}$, we have
\[ f(x_0) \leqslant f(x) + \frac{1}{2\tau}||x-x_0||^2 \quad \textup{on } H \]
with the same equality criterion. \\
Thus the minimizing sequence is constant and so is the minimizing movement.
\end{rmq}

There are lot of results explaining why the minimizing movement corresponds to the classical gradient flow in smooth case, for instance (see \cite[esempio 1.1]{Ambrosio1995}):

\begin{thm} \label{mvt regulier}
Consider a $C^2$ function $f : H \to \RR{} $. If $f$ is Lipschitz continuous or $\displaystyle \lim_{||x|| \rightarrow +\infty } f(x)=+\infty$ then for all starting point $x_0$, $f$ admits a unique minimizing movement starting from $x_0$ given by
\[ \left\{ \begin{array}{rcl}
u'(t) & = & -\nabla f(u(t)) \\
u(0) & = & x_0 . \end{array} \right. \]
\end{thm}

\subsection{Scale invariant Willmore energy.}

We introduce the Willmore energy to which we will apply the minimizing movement principles. Consider $M$ a $C^2$ hypersurface of $\RR{d}$ with $d\geqslant 2$ and its principal curvatures $\kappa_1,\dots,\kappa_{d-1}$. The (scalar) mean curvature on every point $p$ on $M$ is then given by
\[ H(p) = \frac{1}{d-1} \sum_{i=1}^{d-1}\kappa_i(p) . \]
We refer to \cite{Berger1987} for generalities on differential geometry.

\begin{Def}[Scale invariant Willmore energy] \label{def willmore ensemble}
Let $M$ be a $C^2$ hypersurface of $\RR{d}$ and $H$ its mean curvature, the scale invariant Willmore energy of $M$ is given by
\[ W(M) = \int_{M} |H|^{d-1} \ud \HH{d-1}, \]
with $\HH{d-1}$ the $d-1$ Hausdorff measure on $\RR{d}$ or equivalently the area measure on $M$.
\end{Def}

The Willmore energy can be also defined with other exponents than $d-1$, but in that case it is no longer invariant under scaling transformations. The extension to functions of the Willmore energy is straightforward.

\begin{Def} \label{def willmore fonction}
Let $u:\RR{d} \to \RR{}$ be a $C^2$ function, its scale invariant Willmore energy is given by
\[ W_{f}(u) = \frac{1}{(d-1)^{d-1}}\mathlarger\int_{\RR{d}} |\nabla u| \left|\Div \frac{\nabla u}{|\nabla u|}\right|^{d-1} \ud x , \]
with the convention $|\nabla u| \left|\Div \frac{\nabla u}{|\nabla u|}\right|^{d-1} = 0$ if $ |\nabla u|=0$.
\end{Def}

The link between the two definitions follows from the coarea formula which we now recall (see \cite{Evans1992}).

\begin{thm}[General coarea formula] \label{coarea general}
If $f:\RR{n} \to \RR{m}$ is a Lipschitz continuous map with $m \leqslant n$ then for every measurable function $g\in L^1(\RR{n})$ we have
\[ \int_{\RR{m}} \left( \int_{f^{-1}(y)} g(x)\ud \HH{n-m}(x) \right) \ud \HH{m}(y)
 = \int_{\RR{n}}g(x)J_f(x)\ud \HH{n}(x) . \]
where $J_f$ is the generalized Jacobian of $f$.
\end{thm}
Applying this formula to the function
\[ g(x)=\ind{ \{x \ | \ J_f(x)=0 \} } \]
one obtains the following corollary (Federer-Morse-Sard Theorem, see \cite[Remark 2.97]{Ambrosio2000})

\begin{cor}
If $f:\RR{n} \to \RR{m}$ is a Lipschitz continuous map with $m \leqslant n$ then, for $\HH{m}$-almost every $y\in \RR{m}$,
\[ \HH{n-m}\left( \{ x\in \RR{n} \ | \ f(x)=y, J_f(x)=0 \} \right) =0 . \]
\end{cor}

\noindent For $m=1$ and $n=d$ we have $J_f(x)=|\nabla f(x)|$ therefore this corollary implies that, for almost every $t\in \RR{}$, $\{ x \ | \ f(x)=t \}$ is locally a $(d-1)$-rectifiable set.
For convenience, we denote $\{ u\geqslant t\} = \{ x\in\RR{d} \ | \ u(x)\geqslant t \}$.

For $u$ a $C^2$ function, we have $\od \{ u \geqslant t \} \subset \{ u=t \}$. Assume that $|\nabla u| \neq 0$ on $\{ u=t \}$ then this level set is a $C^2$ hypersurface and $\od \{ u \geqslant t \} = \{ u=t \}$. Moreover, its mean curvature vector is given by
\[ \overrightarrow{H} = \frac{-1}{d-1} \Div \left( \frac{\nabla u}{|\nabla u|} \right) \vec{n} \]
where $\vec{n}$ is the normal vector
\[ \vec{n}=\frac{\nabla u}{|\nabla u|} .\]
The mean curvature is simply the scalar:
\[ H = \frac{-1}{d-1} \Div \frac{\nabla u}{|\nabla u|} . \]
By the previous corollary, for any $C^2$ function $u$ and for almost every level set $\{ u = t \}$ there exists a $\HH{n-1}$ negligible set $N$ such that $\od \{ u \geqslant t \} \smallsetminus N $ is a $C^2$ hypersurface and then its Willmore energy can be computed. Summing over all of these level sets and using the general coarea formula one will obtain

\begin{propo}[coarea formula] \label{coarea regulier}
For any $C^2$ function $u:\RR{d} \to \RR{}$, we have
\[ W_f(u) = \int_{-\infty}^{+\infty} W(\od \{u\geqslant t\} ) \ud t . \]
\end{propo}

\subsection{Functional context.}

We consider the $L^1$ relaxation of the Willmore energy in order to extend it to less regular objects and to add suitable functional properties. Let $\OO$ be a bounded open subset of $\RR{d}$.

\begin{Def}[functional Willmore energy] \label{willmore}
For $u\in L^1(\OO)$, its Willmore energy is:
\[ W(u) = \left\{ \begin{array}{ll}
W_f(u) & \qquad \text{if $u$ is $C^2_c(\OO)$} \\
+ \infty & \qquad \text{otherwise} \end{array} \right. \]
where $C^2_c(\OO)$ is the class of $C^2$ functions with compact support in $\OO$. We let $\WW$ be the lower semicontinuous envelop of $W$, i.e. $\WW$ is the relaxed energy:
\[ \WW (u,\OO) = \inf \left\{
\liminf_{n \to +\infty} W(u_n) \quad \big| \quad u_n \longrightarrow  u \quad \text{in } L^1(\OO) \right\} .\]
\end{Def}

As for functions, we define a relaxed energy for subsets of $\RR{d}$.

\begin{Def}[set convergence]
Let $(E_n)$ be a sequence of subsets of $\OO$. We say that $(E_n)$ converges to some subset $E$ in $L^1(\OO)$ if the sequence of characteristic functions $(\ind{E_n})$ converges to $\ind{E}$ in $L^1(\OO)$.
\end{Def}

\begin{rmq}
Recall the symmetric difference between two sets:
\[ E,E' \subset \RR{d} \ : \ E\Delta E' = (E\smallsetminus E' ) \cup ( E' \smallsetminus E) . \]
$(E_n)$ converges to a set $E$ in $L^1(\OO)$ if and only if $|\OO \cap (E_n \Delta E)| \longrightarrow  0$, with $|A|$ the Lebesgue measure on $\RR{d}$ of the subset $A$.
\end{rmq}

\begin{Def}[geometric Willmore energy]
Let $E$ be a subset of $\RR{d}$. We say $E$ is regular if $|E|>0$ and its boundary $\od E$ belongs to $C^2$, i.e. it is locally the graph of a $C^2$ function. We identify the Willmore energy of $E$ to the energy of its boundary. \\
This notion can be extended to every measurable set by relaxing it with respect to the $L^1$ topology. For a measurable set $E$, we define
\[ \WW (E,\OO) = \inf \left\{
\liminf_{n \to +\infty} W(E_n) \quad \big| \quad E_n \longrightarrow  E \quad \text{in } L^1(\OO) \text{ and } E_n \subset \OO \right\} . \]
\end{Def}

By convention $W(\emptyset)=0$ and then $\WW(E,\OO) = 0$ for every negligible set $E$. The relaxation of the Willmore energy has been the purpose of several papers, see \cite{Bellettini1993,Bellettini2004,Leonardi2009,Masnou2013a,Masnou2013}.

\subsection{Study objectives.}

Now we can explain our main goal and how we intend to reach it. Given an initial datum $u_0\in L^1(\RR{d})$, we investigate the convergence of a minimizing sequence of the energy $\WW$ in $L^1$ to a map $t\mapsto u_t \in L^1(\RR{d})$. The general problem is far too hard for now and we focus on some particular initial data: the class of radially non-increasing functions on $\OO$.

\begin{Def}[radially non-increasing function]
$u:\RR{d} \to \RR{}$ is a radially non-increasing function if there exists a function
$r:\RR{}\to\RR{+}$ such that, for almost every $t\in \RR{}$
\[ \{ u\geqslant t \} = \overline{B}(0,r(t))  \]
up to a Lebesgue negligible set. The function $r$ is called the radius function of $u$.
\end{Def}

As we work with functions that are defined almost everywhere, it is convenient to consider only purely radial functions $u(x)=u(|x|)$ and $B(0,r(t)) = \{ u\geqslant t \}$ with $B$ representing either the closed or open ball, depending on $u$ and $t$. All arguments hereafter are valid whenever the sets $B(0,r(t))$ and $ \{ u\geqslant t \}$ coincide up to a negligible set (with respect to the Lebesgue measure of $\RR{d}$). For the sake of simplicity, we will no longer mention this subtlety.

The definition above can be easily extended to a function in $L^1(\OO)$ with support in a ball contained in $\OO$. For the sake of simplicity, we will assume that $\OO = B(0,R_0)$.

We now describe the main direction of the paper. Starting from a radially non-increasing initial datum $u_0 \in L^1(\OO)$, the study of the minimizing movement requires the minimization of the functional
\[ \FF(u)=\WW(u,\OO) + \frac{1}{2\tau}||u-u_0||_{L^1}^2 . \]

As $u_0$ is invariant under rotations of $\RR{d}$, it is natural to think that a minimizer would also be invariant. It is true and the proof is as follows.
\begin{itemize}
\item Consider any candidate $u$ (possibly non radial) and its superlevel sets.
\item Replace each superlevel sets by balls with the same volume.
\item Build the function $\tilde{u}$ having these balls as superlevel sets.
\item Using the coarea formula, compute the Willmore energy of $u$ and $\tilde{u}$ with their superlevel sets.
\end{itemize}
If balls are minimizers of the geometric Willmore energy, we deduce an inequality between Willmore energies of the superlevel sets. This inequality is preserved by coarea formula and then
\[ \WW(\tilde{u},\OO) \leqslant \WW(u,\OO) .\]
By taking balls as superlevel sets, we get closer to the superlevel sets of the initial condition $u_0$. By preserving the volume of the superlevel sets, a classical lemma allows us to compute the $L^1$ norm of $u-u_0$ and $\tilde{u}-u_0$ using the volumes of (almost) every superlevel sets. That means that $\tilde{u}$ is closer from $u_0$ than $u$ is, and then \[ \FF(\tilde{u}) \leqslant \FF(u) .\]
Let us now point out which parts of the proof must be tackled more accurately. The first gap is that Willmore minimizers are not spheres but planes. We want to use compact hypersurfaces to have only spheres as minimizers. Therefore we restrict the problem to a bounded open set $\OO$. The second difficulty is that the relaxed Willmore energy does not satisfy the coarea formula in the general sense, unless assuming that the functions are nonnegative (see the next section). Overall, we will address the problem only for nonnegative and compactly supported functions, and the general problem remains open.

As a result of the above arguments, our minimizers are described only through their radius functions $r$. Studying them becomes easier and it is the subject of the third part. We demonstrate that, in fact, minimizers belong to truncations of the initial radius function $r_0$. Estimations on the best truncated function allow us to iterate the process in order to find a minimizing sequence. This sequence of functions consists of multiple truncations of the initial datum until vanishing. Then we prove that these truncations converge to an erosion of the initial datum, see Figure \ref{intro}. Some properties are finally obtained about the erosion speed.

\begin{rmq}
In this study, the fidelity term in $\FF$ of the minimizing movement is given by a $L^1$ norm:
\[ ||u-u_0||_{L^1} . \]
Using it with the indicator functions $u=\ind{E}$ and $u_0=\ind{E_0}$ yields
\[ d_{L^1} (E,E_0) = |E\Delta E_0| . \]
Although natural, the distance $d_{L^1}$ is actually not suitable for mimicking the classical mean curvature flow (see \cite[chapter 8]{Braides2013}). For instance, the minimizing sequence starting from a circle in the plane is stationary for $\tau$ small enough. To avoid this phenomenon, one rather uses the distance
\[ \delta(E,E_0)= \int_{E \Delta E_0} |d_{E_0}(x)|\ud x \]
with $d_E$ the classical distance or the signed distance to the set $E$ (see \cite{Chambolle2004} and \cite{Luckhaus1995}). \\
In our case, we keep the distance $d_{L^1}$ for two reasons. The first one is that there is no easy coarea formula, equivalent to Lemma \ref{ev-ga} and giving a suitable distance between two functions in terms of the $\delta$-distance between their superlevel sets. The second one is that we work with radial functions, i.e. functions with balls as superlevel sets. The (geometric) scale invariant Willmore energy is minimal for every ball, no matter the radius. Whatever the fidelity term we use on superlevel sets, a minimizer will always be the initial ball. Thus the flow of each superlevel set is stationary.\\
One point of this paper is that even if the superlevel set Willmore flow is stationary, the function flow is non trivial.
\end{rmq}

\section*{Notations}

\begin{tabular}{ll}
$|.|$ & either the absolute value of a real number, \\
 & \qquad the modulus of a vector or the Lebesgue measure on $\RR{d}$ \\
$\alpha_d$ & volume of the unit ball of $\RR{d}$ \\
$\omega_d$ & area of the unit sphere of $\RR{d}$ \\
$A\Delta B$ & symmetric difference between the sets $A$ and $B$ \\
$\HH{k}$ & $k$-dimensional Hausdorff measure
\end{tabular}

\section*{Acknowledgement}

My special thanks go to Simon Masnou for all the stimulating discussions during the preparation of this paper.

\section{Coarea formula for Willmore energy.} \label{partie coarea}

In this section, we will discuss the coarea formula for Willmore energy (non necessarily scale invariant).
Let $p\in [1,+\infty[$, we extend Definitions \ref{def willmore ensemble} and \ref{def willmore fonction} to curvature with power $p$, with the same conventions.

\begin{Def}[$p$-Willmore energies.]
Let $u$ be a $C^2(\RR{d},\RR{})$ function and $E$ be a $C^2$ subset of $\RR{d}$. We set
\[ W_p(u) = \frac{1}{(d-1)^p} \mathlarger\int_{\RR{d}} |\nabla u| \left|\Div \frac{\nabla u}{|\nabla u|}\right|^p \ud x \] and
\[ W_p(E) = \int_{\od E} |H|^p \ud \HH{d-1} .\]
We define also the relaxation of these energies for every function $u\in L^1(\OO)$ and every Borel set $E\subset \OO$:
\[ \Wp (u,\OO) = \inf \left\{
\liminf_{n \to +\infty} W_p(u_n) \quad \big| \quad u_n \in C^2_c(\OO),
	\ u_n \longrightarrow  u \quad \text{in } L^1(\OO) \right\}  \] and
\[ \Wp (E,\OO) = \inf \left\{
\liminf_{n \to +\infty} W_p(E_n) \quad \big| \quad E_n \subset \OO, \ \od E \text{ in } C^2, \ E_n \longrightarrow  E \quad \text{in } L^1(\OO)  \right\} . \]
\end{Def}

\noindent Notice that $W_{d-1}=W$ with the previous notation \ref{willmore}.

The purpose here is to prove the following coarea formula for $\Wp$:
\[ \int_{-\infty}^{+\infty} \Wp(\{u\geqslant t\},\OO) \ud t \quad = \quad \Wp(u,\OO) .\]
In general settings such a formula is not valid, see \cite{Masnou2013}. In our special case of radially non-increasing functions, we will prove that it is true.

We will first state four lemmas whose proofs are given in the appendix. The first one is a coarea inequality proved in \cite[Theorem 4 and Remark 2]{Ambrosio2003}.
\begin{lem}[coarea inequality] \label{inegalite}
Let $u\in L^1(\OO)$, we have
\[ \int_{-\infty}^{+\infty} \Wp(\{u\geqslant t\},\OO) \ud t \quad \leqslant \quad \Wp(u,\OO) .\]
\end{lem}

\noindent The second one is a classical lemma based on Cavalieri formula.

\begin{lem}[$L^1$ norm] \label{ev-ga}
For measurable functions $u,v : \RR{d} \mapsto \RR{}$ we have
\[ \int_{\RR{d}} |u-v|\ud x =
\int_{-\infty}^{+\infty} |\{u\geqslant t \} \Delta \{ v\geqslant t\}| \ud t \]
with $\{u\geqslant t\}$ the $t$-height superlevel set of $u$.
\end{lem}

\noindent The third one is only required for technical purpose.

\begin{lem}[Reattachment] \label{recollement}
Let $a,b,\alpha,\beta \in \RR{}$ such that $a,b<0$. There exists a decreasing function $f\in C^2([0,1])$ such that
\[ \begin{array}{cc}
f(0)=1 & f(1)=0 \\ f'(0)=a & f'(1)=b \\ f''(0)=\alpha & f''(1)=\beta . \end{array} \]
\end{lem}

\noindent The last one is the equality of $\Wp$ and $W_p$ on regular subsets of $\OO$.

\begin{lem}[Relaxation on regular sets] \label{regulier}
Let $p>1$ and let $E$ be a $C^2$ compact subset included in $\OO$. Then
\[ \Wp(E,\OO) = W_p(E) .\]
\end{lem}

Now we state and prove the main result of this section.

\begin{thm} \label{coarea}
Let $v: \RR{d} \to \RR{}$ be a radially non-increasing function which is non-negative, bounded and compactly supported in $\OO=B(0,R_0)$. For $p>1$, we have
\[ \int_{-\infty}^{+\infty} \Wp(\{v\geqslant t\},\OO) \ud t \quad = \quad \Wp(v,\OO) .\]
\end{thm}

\noindent For such function $v$, we denote by $r:\RR{} \to \RR{+}$ the radius function defined by the relation
\[ \{v\geqslant t\} = B(0,r(t)).\]
Note that $r(0)<+\infty$ is the radius of the support of $v$, $r(t)>0$ for all $t$ in $[0,\sup v[$ and $r(t)=0$ for $t>\sup v$. Using Lemma \ref{regulier} and the fact that the mean curvature of a ball with radius $r$ is $H=1/r$, we have
\[ \Wp(\{v\geqslant t\},\OO) = \omega_d r(t)^{d-1-p}, \]
and then
\[ \Wp(v,\OO) = \omega_d \int_0^{\sup v} r(t)^{d-1-p} \ud t \]
with $\omega_d$ the area of the unit sphere of $\RR{d}$.

\begin{proof}
Let $r$ be the previous radius function, we denote $b=\sup v < +\infty$.

\textbf{ Particular case. } \\
We assume some hypotheses on $r$ in order to have a regular enough $v$ such that the classical coarea formula applies on $W_p$.
Suppose $r$ is a $C^2$ decreasing function on $]0,b[$ and there exist $\ee>0$ and $q \in ]0,1/4[$ such that
\[ \forall t \in [b-\ee,b], \quad r(t) = (b-t)^q \]
\[ \forall t \in [0,\ee], \quad r(t) = r(0)-t^q. \]
In that case, $r$ is one-to-one and its derivative does not vanish on $]0,b[$. Then we have $v(x)=r^{-1}(|x|)$ and the function $v$ is $C^2$ on $B(0,r(0)) \smallsetminus \{ 0 \}$. However, for $|x|<r(b-\ee)$, we have $v(x) = b-|x|^{1/q} = b-(|x|^2)^{1/2q}$, which is $C^2$ on $x=0$ because $1/2q > 2$. In the same way, for $r(\ee) < |x| < r(0)$, we have $v(x) = (r(0)-|x|)^{1/q}$ which links to the zero function with a $C^2$ regularity on $\RR{d} \smallsetminus B(0,r(0))$. This allows $v \in C^2_c(\OO)$ and using the coarea formula (proposition \ref{coarea regulier}), we get \[ W_p(v) = \int_0^{+\infty} W_p(\{ v \geqslant t\}) \ud t .\]

\textbf{General case.}  \\
In general $r$ is a non increasing function, see Figure \ref{graph} for a typical radius function example.

\begin{figure}[!ht]

\begin{center}
\begin{tikzpicture}[scale=0.75]
\draw (-4,0) node[below left] {$0$} ;
\draw [->] (-5,0) -- (7,0) node[below right] {$t$} ;
\draw [->] (-4,-1) -- (-4,4) node[left] {$r$} ;
\draw (-5,3) -- (7,3) node[near end, above] {$r(0)$} ;
\draw [thick] (-4.5,3) -- (-3,3) 
plot [domain=-3:1,samples=100] (\x,{3+(2/pi)*(rad(atan(-1*\x))-rad(atan(3)))})
plot [domain=1:3/2] (\x,{sqrt(19/10-\x)*sqrt(10)/2}) -- (2,1)
plot [domain=2:5,samples=100] (\x,{(5-\x)^2/9}) node {$|$} -- (6,0) ;
\draw [thick,dotted] (1,1.7) -- (1,1.5) ;
\draw (5,-1/2) node {$b$} ;
\end{tikzpicture}
\end{center}

\caption{Graph of the function $r$.}
\label{graph}

\end{figure}
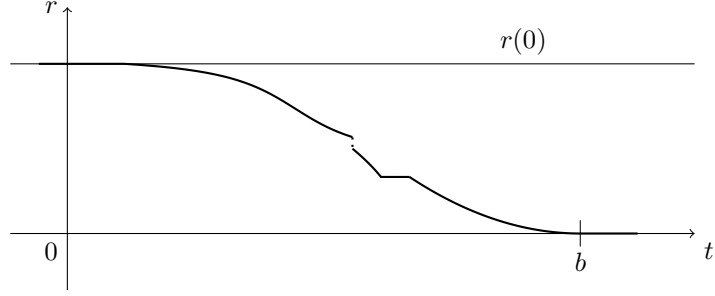

\noindent We will approximate the function $r$ with a regular enough sequence to use the previous case, and with a good enough convergence to pass to the limit.
We obtain regularity on $r$ using a well chosen convolution. Indeed, the Willmore energy of a ball with radius $r$ is (up to a constant) $r^{d-1-p}$ and, when $d-1-p<0$, this means the energy sequence is the integral of the inverse-like functions $t \mapsto 1/r_n(t)^{|d-1-p|}$. For the energy to converge, we will use the monotone convergence theorem, this is why we need a non increasing sequence $(r_n)$. In the case $d-1-p\geqslant 0$, the convergence is straightforward, because the sequence is bounded and with compact support. Finally, we adjust this $C^2$ sequence of functions to add the $t^q$ tails.

\textbf{Step 1: } let $\rho : \RR{} \to \RR{}$ be a mollifier such that $\rho$ is $C^2$, $\rho > 0$ on $]0,1[$, supp$(\rho) = [0,1]$ and $\displaystyle \int_{\RR{}} \rho (t)\ud t = 1$. Fix $R \in ]r(0),R_0[$ and extend $r$ by $R$ for $t<0$ and by $0$ for $t>b$ (taking such a $R$ will be important later in order to have decreasing  -- and not only nonincreasing -- functions $r_n$). For $n\in \NN^*$ we define $\rho_n(t) = n\rho(nt)$ and $\tilde{r}_n = r*\rho_n$. As $r$ is bounded and $\rho$ is compactly supported, $\tilde{r_n}$ is $C^2$ on $\RR{}$. Moreover, we have
\[ \tilde{r}_n(t) = \int_{\RR{}}nr(t-s)\rho(ns)\ud s = \int_0^{1/n}nr(t-s)\rho(ns)\ud s = \int_0^1 r(t-s/n)\rho(s) \ud s . \]

Thanks to that formula, we can prove that $(\tilde{r}_n)$ is a non increasing sequence of non increasing functions. For all $t \in \RR{}$ and for all $s \in [0,1]$,
\[ t-\frac{s}{n} \leqslant  t-\frac{s}{n+1}
\quad \Rightarrow \quad
r\left( t-\frac{s}{n+1} \right)\rho(s) \leqslant r\left( t-\frac{s}{n} \right)\rho(s) .\]
Integrating with respect to $s$, we have $\tilde{r}_{n+1}(t) \leqslant \tilde{r}_{n}(t) $. So $(\tilde{r}_n)$ is a non increasing sequence. \\
Similarly for all $t\leqslant t'$ and $s \in [0,1]$,
\[ t-\frac{s}{n} \leqslant t'-\frac{s}{n}
\quad \Rightarrow  \quad 
r\left( t'-\frac{s}{n} \right)\rho(s) \leqslant r\left( t-\frac{s}{n} \right)\rho(s) . \]
So $\tilde{r}_{n}(t') \leqslant \tilde{r}_{n}(t)$ and then $\tilde{r}_n$ is a non increasing function. \\
Let us now define
\[ b_n = \inf\left( t \in \RR{} \ | \ \tilde{r}_n(t) = 0 \right) \]
and check that $b_n = b + \frac{1}{n}$.
Indeed, for all $t>b+1/n$ and $s \in [0,1]$ we have
\[ t-\frac{s}{n} > b \quad \Rightarrow \quad r\left(t-\frac{s}{n}\right)=0 \quad \Rightarrow \quad \tilde{r}_n(t) = 0 .\]
So $b_n \leqslant b + 1/n$. For all $t \in ]b,b+1/n[$, there exists a $s_0\in ]0,1[$ such that $t-s_0/n = b$. Hence, $r(t-s/n)=0$ on $[0,s_0[$ and $r(t-s/n)>0$ on $]s_0,1]$. Thus $\tilde{r}_n(t) > 0$. As $\tilde{r}_n$ is non increasing, we have
\[ \forall t < b +1/n, \quad \tilde{r}_n(t)>0 . \]
So $b_n \geqslant b +1/n$. \\
Moreover, for all $t>0$, $\tilde{r}_n(t) < R$. Indeed for $t \in ]0,1/n[$ , there exists a $s_0\in ]0,1[$ such that $t-s_0/n = 0$. Thus, $r(t-s/n)\leqslant r(0)<R$ on $[0,s_0[$ and $r(t-s/n)=R$ on $]s_0,1]$. Hence $\tilde{r}_n(t) <R$. As $\tilde{r}_n$ is non increasing, the inequality remains valid for larger $t$.

\textbf{Step 2: } now we modify this sequence in order to have decreasing functions, to be able to add the $t^q$ tails after $t=b_n$ and at $t=0$ and then to have a $C^2$ function $v_n$. Let $p^* = |d-1-p|$  if $d-1-p<0$ and $p^*=1$ else and take the following auxiliary function (see Figure \ref{regul 1}) defined by
\[ \forall t \in [0,b_n], \qquad h_n(t) =
\left( 1 - \frac{1}{n^{1/2p^*}}\right)\frac{t}{b_n} . \]
Then define
\[ \forall t \in [0,b_n], \qquad r_n(t) = R - h_n(t)(R- \tilde{r}_n(t)) . \]

\begin{figure}[!ht]

\begin{center}
\begin{tikzpicture}[scale=0.9]
\draw (-4,0) node[below left] {$0$} ;
\draw [->] (-4.5,0) -- (8.2,0) node[below right] {$t$} ;
\draw [->] (-4,-1) -- (-4,4) ;
\draw (-4.5,3) -- (7,3) ; \draw (7.4,3) node {$r(0)$} ;
\draw (-4.5,3.42) -- (7,3.42) ; \draw (7.2,3.42) node {$R$} ;
\draw
plot [domain=-3:1,samples=100] (\x,{2.99+(2/pi)*(rad(atan(-1*\x))-rad(atan(3)))})
plot [domain=1:3/2] (\x,{sqrt(19/10-\x)*sqrt(10)/2}) -- (2,1)
plot [domain=2:5,samples=100] (\x,{(5-\x)^2/9}) node {$|$} ;
\draw [dotted] (1,1.7) -- (1,1.5) ;
\draw [thick] plot [domain=-4:6,samples=100] (\x,{1.665-3*1.3/pi*rad(atan(\x-1.8))})
node {$|$} ;
\draw [thick] plot [domain=-4:6,samples=100]
(\x,{3.4-(\x+4)*0.075*(4-1.665+3*1.3/pi*rad(atan(\x-1.8)))}) ;
\draw [dotted] (5.85,-0.2) rectangle (8,0.7) ;
\draw [dotted] (-4.15,3.8) rectangle (-2,3.1) ;
\draw (7.2,0.3) node {$t^q$ tail} ;
\draw (-2.7,3.6) node {$t^q$ tail} ;
\draw [dashed] (0,0.40) -- (6.5,0.40) ;
\draw [<->] (5,-0.75) -- (6,-0.75) ;
\draw [<->] (4.6,1.05) -- (4.6,3.395) ;
\draw (4.2,0.98) -- (4.8,0.98) ;
\draw [<->] (4.5,3.395) -- (4.5,0.25) ;
\draw [<->] (-0.1,0.05) -- (-0.1,0.40) ;
\draw (5.5,-1) node {$1/n$} (6.1,2)
node {$h_n(t)(R-\tilde{r}_n(t))$} (-0.9,0.2) node {$R/n^{1/2p^*}$};
\draw (5,-1/2) node {$b$} (6,-1/2) node {$b_n$} ;
\draw (0.75,2) node {$r$} (2.2,1.5) node {$\tilde{r_n}$} (2.85,2) node {$r_n$} ;
\end{tikzpicture}
\end{center}

\caption{Regularization of $r$ on $[0,b_n]$.}
\label{regul 1}

\end{figure}
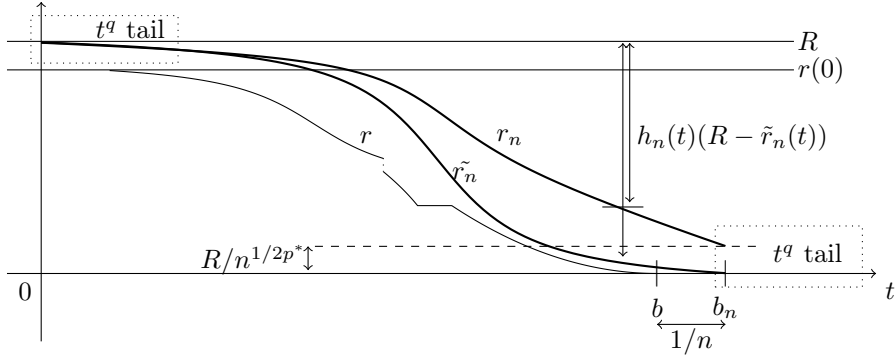

\noindent This sequence remains non increasing. Indeed, picking $t$ in $[0,b_n]$, $(b_n)$ is non increasing and non negative so $\left(\frac{t}{b_n}\right)$ is also non increasing and non positive. $\left(1-\frac{1}{n^{1/2p^*}}\right)$ is non increasing and non negative so $(h_n(t))$ is non increasing and non negative. $(\tilde{r}_n(t))$ is non increasing and bounded by $R$, hence $((R-\tilde{r}_n(t))h_n(t))$ is non decreasing and $(r_n(t))$ is non increasing. Furthermore $0\leqslant h_n(t) \leqslant 1-1/n^{1/2p^*}$, therefore the sequence is bounded by $R/n^{1/2p^*}$ and $R$. \\
Moreover, we easily find that $r_n$ is decreasing thanks to the calculation
\[ \forall t \in ]0,b_n], \quad r'_n(t) = \underbrace{-\left( 1-\frac{1}{n^{1/2p^*}}\right)}_{<0}
\left( \underbrace{\frac{1}{b_n}}_{>0}
\underbrace{(R-\tilde{r}_n(t))}_{>0}
+ \underbrace{\frac{t}{b_n}}_{>0}
\underbrace{(-\tilde{r}'_n(t))}_{\geqslant 0}
\right) < 0 . \]

\textbf{Step 3: } now we build the $t^q$ tail beyond $t=b_n$. With $0<q<\min(1,1/2p^*)$, consider
\[ \forall t \in [c_n,b_n+1/n], \quad \eta(t) = \left( b+\frac{2}{n}-t\right) ^q . \]
We took
\[ c_n = b+\frac{2}{n}-\left( \frac{R}{2} \right)^{1/q} \frac{1}{n^{1/2pq}} , \]
with $c_n$ chosen to satisfy $\eta(c_n)=\frac{R}{2n^{1/2p^*}}$, and for $n>(R/2)^\frac{1/q}{1/2p^*-1}$, $b_n<c_n<b_n+1/n$. Also
\[ r_n(b_n)=\frac{R}{n^{1/2p^*}} \quad > \quad \frac{R}{2n^{1/2p^*}} = \eta(c_n), \]
\[ r'_n(b_n)=-R\left( 1- \frac{1}{n^{1/2p^*}} \right) < 0 \quad \text{and} \quad
\eta'(c_n)=-q\left( \left( \frac{R}{2} \right)^{1/q} \frac{1}{n^{1/2p^*}} \right)^{q-1} < 0. \]
Thanks to Lemma \ref{recollement}, we can reattach $\eta$ and $r_n$ between $[b_n,c_n]$ in a $C^2$ and decreasing way (see Figure \ref{tail 1}).

\begin{figure}[!ht]

\begin{center}
\begin{tikzpicture}
\draw [->] (-1,0) -- (6,0) node[above] {$t$} ;
\draw [dashed] (0,0) -- (0,3) ;
\draw [dashed] (3,0) -- (3,3) ;
\draw [dashed] (-0.2,1.25) -- (6,1.25) ;
\draw [dashed] (-1,2.5) -- (6,2.5) ;
\draw [thick] plot [domain=-1:0] (\x,{2.5 +0.9/sqrt(2)*(sqrt(1-\x)-1)}) ;
\draw [thick] plot [domain=3:5,samples=50] (\x,{1.25/sqrt(2)*sqrt(5-\x)}) ;
\draw [<->] (-0.3,0.05) -- (-0.3,2.45) ;
\draw [<->] (0.2,0.05) -- (0.2,1.20) ;
\draw (1.5,1.9) node {$C^2$ junction} ;
\draw (0,-0.5) node {$b+1/n$} (3,-0.5) node {$c_n$} (5.2,-0.5) node {$b+2/n$} ;
\draw (-0.9,1.25) node {$\displaystyle \frac{R}{n^{1/2p^*}}$} (1.1,0.6) node {$R/2n^{1/2p^*}$} ;
\draw (4.7,0.8) node {$\eta$} (-0.3,2.8) node {$r_n$} ;
\end{tikzpicture}
\end{center}

\caption{Regularization of $r$ near $b$.}
\label{tail 1}

\end{figure}
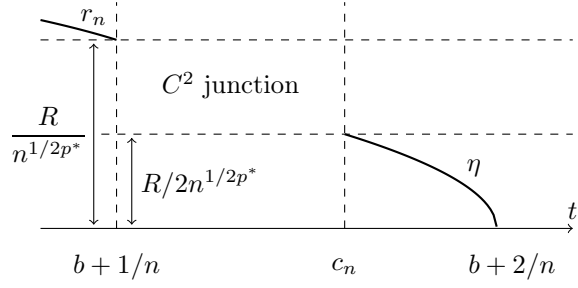

\noindent Now we extend $r_n$ by $0$ beyond $b_n+1/n$, which allows it to be in the regular case.
\medskip

\textbf{Step 4: } once more we add the $t^q$ tail near $t=0$ considering
\[ \forall t \in [0,1/n], \quad \theta(t) = R+\frac{1}{n}-t^q . \]
Let $a_n=(1/2n)^{1/q}$ be such that $\theta(a_n)=R+1/2n$. Note we have $0<a_n<1/n$ when $n>(1/2)^{\frac{1/q}{1/q-1}}$. Furthermore
\[ r_n(1/n) < R < \theta(a_n), \quad r_n'(1/n)<0 \quad \text{and} \quad
\theta'(a_n)=-q\left( \frac{1}{2n} \right)^\frac{q-1}{q} < 0. \]
Then, as before, we can do a $C^2$ reattachment using Lemma \ref{recollement} (see Figure \ref{tail 2}).

\begin{figure}[!ht]

\begin{center}
\begin{tikzpicture}
\draw [->] (0,0) -- (0,3.5) ;
\draw (-0.5,1.5) -- (6.5,1.5) ; \draw (6.7,1.4) node {$R$} ;
\draw [dashed] (-0.2,2.25) -- (6.5,2.25) ;
\draw [dashed] (-0.2,3) -- (6.5,3) ;
\draw [dashed] (1.5,3) -- (1.5,1.4) ;
\draw [dashed] (4,3) -- (4,0.7) ;
\draw (3.75,1) -- (4.25,1) ;
\draw [thick] plot [domain=0:6.5,samples=100] (\x,{1.48-0.03*(\x)^2}) ;
\draw [thick] plot [domain=0:1.5,samples=100] (\x,{3-0.6*sqrt(\x)}) ;
\draw [<->] (5.5,2.9) -- (5.5,1.55) ;
\draw [<->] (5.65,2.15) -- (5.65,1.55) ;
\draw (5.1,2.6) node {$1/n$} (6.2,1.8) node {$1/2n$} ;
\draw (1.5,1.2) node {$a_n$} (4,0.5) node {$1/n$} ;
\draw (0.8,2.75) node {$\theta$} (5.5,0.75) node {$r_n$} ;
\draw (2.75,1.75) node {$C^2$ junction} ;
\end{tikzpicture}
\end{center}

\caption{Regularization of $r$ near $a$.}
\label{tail 2}

\end{figure}
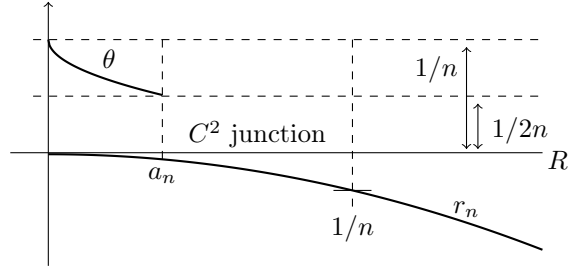

\noindent The sequence $(r_n)$, built with the two tails, satisfies the assumption of the regular case, therefore we have \[ W_p(v_n) = \int_0^{\sup v_n} W_p(B(0,r_n(t)) \ud t
= \omega_d \int_0^{b+2/n} r_n(t)^{d-1-p} \ud t . \]
We just have to check that $v_n$ converges to $v$ in $L^1$ and that the integral above converges to the corresponding integral with $r$.
\medskip

\textbf{Step 5: } thanks to the convolution, $\tilde{r}_n$ converges to $r$ (expanded by $R$ and $0$) in $L^1_{loc}$. Thus, in the compact $[0,b+2]$, up to a subsequence,
\[ \text{for almost every } t\in [0,b+2], \quad \tilde{r}_n(t) \to r(t) \qquad \text{when } n\to +\infty.\]
Moreover, for all $t\in ]0,b]$, there exists a $n_0 \in \NN$ such that, for all $n>n_0$, $1/n<t$ and then, as $t$ is far away from the tail in $0$, 
\[ |r_n(t)-\tilde{r}_n(t)|
\leqslant |R-\tilde{r}_n(t)|\left| \frac{t}{b+1/n} \right| \frac{1}{n^{1/2p}}
\leqslant \frac{R}{n^{1/2p}} \stackrel{n\to +\infty}{\longrightarrow} 0 . \]
For all $t>b$, there exists a rank $n$ such that $b+2/n<t$ and so
\[ r_n(t) = \tilde{r}_n(t) = r(t) = 0. \]
Thus, $r_n$ converges to $r$ almost everywhere on $[0,b+2]$.
Take $v_n(x)=r^{-1}_n(|x|)$ and using lemma \ref{ev-ga}
\[ ||v-v_n||_{L^1} = \int_0^{b+2} \alpha_d|r_n(t)^d - r(t)^d| \ud t  \]
with $\alpha_d$ the volume of the unit ball of $\RR{d}$. \\
As $|r_n(t)^d - r(t)^d|$ converges to $0$ almost everywhere and $|r_n(t)^d - r(t)^d| \leqslant 2R^d$ then, by the Lebesgue dominated convergence theorem,
\[ v_n \stackrel{n\to +\infty}{\longrightarrow} v \qquad \text{in } L^1(\RR{d}) . \]

\textbf{Step 6: } if $d-1-p\geqslant 0$ then, as $|r_n|\leqslant R$, by the Lebesgue dominated convergence theorem, we have
\[ W_p(v_n) = \omega_d \int_0^{b+2/n} r_n(t)^{d-1-p} \ud t
\stackrel{n\to +\infty}{\longrightarrow}
\omega_d \int_0^{b} r(t)^{d-1-p} \ud = \int_0^{+\infty} W_p(\{ v\geqslant t\}) \ud t .\]
If $d-1-p<0$ then we split the integral on several intervals
\[ \frac{W_p(v_n)}{\omega_d} = \int_0^{b+2/n} \frac{1}{r_n(t)^{p^*}} \ud t = 
\underbrace{\int_0^{1/n} \frac{1}{r_n(t)^{p^*}} \ud t}_{I_1} +
\underbrace{\int_{1/n}^b \frac{1}{r_n(t)^{p^*}} \ud t}_{I_2} +
\underbrace{\int_b^{c_n} \frac{1}{r_n(t)^{p^*}} \ud t}_{I_3} +
\underbrace{\int_{c_n}^{b+2/n} \frac{1}{r_n(t)^{p^*}} \ud t}_{I_4}. \]
Note that after both $C^2$ reattachments the sequence $(r_n)$ may no longer be globally decreasing. However we only need this monotonicity to be true on $[1/n,b]$ which is the case. \\
For $I_1$, for all
$t\in [0,1/n]$: $r_n(t)\geqslant r_n(1/n) \geqslant r_n(b/2) \geqslant r(b/2)>0$ because $r_n$ is a non-increasing function and $(r_n)$ a non-increasing sequence bounded by below by $r$. So
\[ \forall t \in [0,1/n], \ \frac{1}{r_n(t)^{p^*}} \leqslant \frac{1}{r(b/2)^{p^*}}. \]
Then $I_1 \to 0$ when $n\to +\infty$. \\
For $I_2$, $\displaystyle \left( \ind{t>1/n}\frac{1}{r_n(t)^{p^*}}  \right) $ is a non-decreasing sequence of functions converging almost everywhere to $\frac{1}{r(t)^{p^*}}$ on $]0,b[$. By the monotone convergence theorem,
\[ I_2 \stackrel{n\to +\infty}{\longrightarrow} \int_0^{b} \frac{1}{r(t)^{p^*}} \ud t .\]
For $I_3$, on $[b,c_n]$: $r_n(t) \geqslant R/n^{1/2p^*}$, so
\[ I_3 \leqslant \frac{\sqrt{n}}{R^{p^*}}(c_n-b) \leqslant \frac{2}{R^{p^*}}\frac{1}{\sqrt{n}} .\]
Then $I_3 \to 0$ when $n\to +\infty$. \\
For $I_4$, on $[c_n,b+2/n]$:
\[ r_n(t) = \left( b+\frac{2}{n}-t \right)^q .\]
Using the Cauchy-Schwarz inequality (and writing
$A=\left( \frac{R}{2} \right)^{1/q}\frac{1}{n^{1/2qp^*}}$) we have
\[ I_4 = \int_0^A \frac{\ud s}{s^{qp^*}} \leqslant
\sqrt{\int_0^1 \ind{s<A}^2 \ud s}\sqrt{\int_0^1 \frac{\ud s}{s^{2qp^*}} } . \]
The first square root is equal to $\displaystyle \left( \frac{R}{2} \right)^{1/2q}\frac{1}{n^{1/4qp^*}} \to 0$ when $n\to +\infty$. The second one is finite because $2qp^*<1$. Then $I_4 \to 0$ when $n\to +\infty$. \\
Finally we have, in every case,
\[ W_p(v_n) = \omega_d\int_0^{b+2/n} r_n(t)^{d-1-p} \ud t 
\stackrel{n\to +\infty}{\longrightarrow}
\omega_d\int_0^b r(t)^{d-1-p} \ud t = \int_0^{+\infty} \Wp(\{v\geqslant t\},\OO) \ud t .\]
As $v_n\to v$ in $L^1$ and by the definition of $\Wp$ we have
\[ \Wp(v,\OO) \leqslant \int_0^{+\infty} \Wp(\{v\geqslant t\},\OO) \ud t . \]
The other inequality is always true (lemma \ref{inegalite}), consequently
\[ \Wp(v,\OO) = \int_0^{+\infty} \Wp(\{v\geqslant t\},\OO) \ud t . \]
\end{proof}

\begin{rmq}
Actually we proved that for any $p>0$
\[ \Wp(v,\OO) \leqslant \int_0^{+\infty} W_p(\{ v \geqslant t \}) \ud t . \]
The restriction $p>1$ is required only for Lemma \ref{regulier}, in order to identify $\Wp$ with $W_p$ for balls and to have the equality . We will apply Theorem \ref{coarea} for $p=d-1$, which does not satisfy $p>1$ for $d=2$. However, for the invariant case $p=d-1$, balls are absolute minimisers of $W_p$ among regular compact sets, and thus Lemma \ref{regulier} can be replaced by Lemma \ref{minwillmore} to prove Theorem \ref{coarea}.
\end{rmq}

\begin{rmq}
Theorem \ref{coarea} can easily be extended in any open bounded subset $\OO$ for a radial function $v$ centered at some $x$, such that supp$(v)=B(x,R) \subset \OO$.
\end{rmq}

\begin{rmq}
We can have a slightly better result by replacing $v$ bounded by
\[ \int_0^{+\infty} r(t)^{d-1-p} \ud t < +\infty \]
and $v \in L^1$. Indeed, take the sequence $v_n$ by cutting $v$ where its values are over $n$. The corresponding radius sequence are $t \mapsto r(t)\ind{t<n}$. Apply Theorem \ref{coarea} to $v_n$, the Willmore energies converge to the right limit thanks to the new hypothesis.
\end{rmq}

\section{Minimizing movement of radial functions.} \label{partie mouvement}

Let $\OO = B(0,R_0)$ and consider the functional space
\[ L^1_{c+} (\OO) = \left\{ u\in L^1(\RR{d}) \ \big| \ u\geqslant 0,
\text{ supp}(u) \subset \OO  \right\} . \]
and let $u_0\in L^1_{c+} (\OO)$ be a radially non-increasing initial datum. We describe the minimizing movement of $\WW$ on $L^1_C$ in two steps. First we study the minimization of
\[ \FF (u) = \WW (u,\OO) + \frac{1}{2\tau} ||u-u_0||_{L^1}^2  \]
in $L^1_{c+} (\OO)$ and we let $u_1$ be a minimizer. Then we iterate the scheme replacing $u_0$ by $u_1$ to construct a minimizing sequence $(u_n)$. \\
In a second step, we study the behaviour of the minimizing sequence when $\tau \to 0$ by writing $u_\tau (t) = u_{[t/\tau]}$ and looking for a limit function  $u : [0,+\infty[ \to L^1_{c+} (\OO)$ of $u_\tau$ when $\tau \to 0$, maybe up to a subsequence.

We will describe the minimizing movement throughout three results. The first one confirms the natural intuition that there exists a radially non-increasing minimizer of $\FF$, i.e. it has the same structure as the initial datum. Our second result states that this minimizer is a truncation of the initial datum (all values larger than some $\lambda$ are replaced by $\lambda$). Then we prove that the minimizing movement is an erosion of the initial datum whose speed is given by an ordinary differential equation.

We denote by $r_0$ the radius function of the initial datum $u_0$.

\subsection{Radially non-increasing minimizer.}

First we prove that among all functions of $L^1_{c+} (\OO)$, the radially non-increasing ones are the best for our minimization problem, see Theorem \ref{radialisation} below.

Consider any candidate $u\in L^1_{c+} (\OO)$. Replace every superlevel set by a ball with the same volume. Doing that, one constructs a radially non-increasing function $\tilde{u}$ with a lower Willmore energy (using the previous coarea formula \ref{coarea})and, as we will see, closer to the initial datum. That is why
\[ \FF(\tilde{u}) \leqslant \FF(u) . \]

Let us now write all the details. Let $u$ be a candidate function in $L^1_{c+} (\OO)$, define the radius function of $u$ by the relation
\[ |\{ u\geqslant t\}| = |B(0,r(t))| . \]
$r$ is a non-increasing function such that $\textup{supp }r = [0,\sup u]$ and $r(0) < +\infty \text{ because } |B(0,r(0))| = | \textup{supp }u |$.
Then we define
\[ \tilde{u}(x) = \sup \{ t\in \RR{} \ | \ r(t) \geqslant |x| \} \]
such that
\[ \{ \tilde{u} \geqslant t\} = \overline{B}(0,r(t)) . \]
$\tilde{u}$ is a radially non-increasing function with $r$ as radius function. \\
Use Lemma \ref{ev-ga} to see that
\[ ||\tilde{u} ||_{L^1} = \int_0^{+\infty} |\{ \tilde{u} \geqslant t\}| \ud t
= \int_0^{+\infty} |B(0,r(t))| \ud t
= \int_0^{+\infty} |\{ u \geqslant t\}| \ud t = ||u||_{L^1} . \]
Now we use the following lemmas (see the appendix for proofs)

\begin{lem} \label{inclusion}
Let $\lambda,\mu>0$. The function $(A,B) \mapsto |A\Delta B|$, defined on the collection of pairs $(A,B)$ such that $|A|=\lambda$ and $|B|=\mu$, reaches its minimum whenever $A\subset B$ or $B\subset A$ (up to a negligible set) and its minimum is $|A\Delta B| = |\lambda-\mu|$.
\end{lem}

\begin{lem} \label{minwillmore}
Let $E$ be a compact and non negligible subset of $\OO$. Then
\[ \WW(E,\OO) \geqslant \omega_d. \]
\end{lem}

On one hand $\{ u\geqslant t\}$ and $\{ \tilde{u}\geqslant t\}$ have the same volume. On the other hand $\{ \tilde{u}\geqslant t\}$ and $\{ u_0\geqslant t\}$ are concentric balls. Thus, using Lemma \ref{inclusion}
\[ |\{ u\geqslant t\} \Delta \{ u_0\geqslant t\}| \geqslant 
|\{ \tilde{u} \geqslant t\} \Delta \{ u_0\geqslant t\}| . \]
Integrating over $t$ on $[0,+\infty[$ and using Lemma \ref{ev-ga}, we have
\begin{equation} \label{min1} ||u-u_0||_{L^1} \geqslant ||\tilde{u}-u_0||_{L^1} . \end{equation}
Moreover, for all $t>0$, $\{ u\geqslant t\}$ is compact, so with the Lemma \ref{minwillmore} we have $|\{ u\geqslant t\}|=0$ or $\WW(\{ u\geqslant t\},\OO) \geqslant \omega_d$.
In the first case, that means $r(t)=0$ and then $\WW(B(0,r(t)),\OO) = 0$. In the second case, we observe that $\omega_d = \WW(B(0,R),\OO) = W(B(0,R))$ for all $R>0$. In any case we get
\[ \WW(\{ u\geqslant t\},\OO) \geqslant \WW(\{ \tilde{u}\geqslant t\},\OO) . \]
Integrating over $t$ on $[0,+\infty[$, using the coarea formula (Theorem \ref{coarea}) for $\tilde{u}$ and the coarea inequality (Lemma \ref{inegalite}) for $u$, we have
\begin{equation} \label{min2} \WW(u,\OO) \geqslant \WW(\tilde{u},\OO) . \end{equation}
By combining (\ref{min1}) and (\ref{min2}) we finally prove the following result.

\begin{thm} \label{radialisation}
For all $u\in L^1_{c+} (\OO)$, there exists $\tilde{u}\in L^1_{c+} (\OO)$ radially non-increasing such that \[ \FF(\tilde{u}) \leqslant \FF(u) ,\]
both values being possibly $+\infty$.
\end{thm}

\begin{rmq}
With the idea that the ball minimizes the perimeter with prescribed volume, we can prove with the coarea formula for functions with bounded variation that this construction also decreases the total variation. Indeed we have
\[ TV(\tilde{u},\OO) \leqslant TV(u,\OO) .\]
\end{rmq}

\subsection{Shape of a minimizer.} \label{forme min}

We now know that we can look for the minimum of $\FF$ among radially non-increasing functions of $L^1_{c+} (\OO)$. We transform the minimization problem on $\FF$ onto a minimization problem on the radius. We introduce
\[ \mathcal{A} = \{ r:[0,+\infty[ \to \RR{+} \ | \ r \text{ is non-increasing, }
r(0)< R_0 \text{ and } r\in L^d([0,+\infty[) \} . \]
Any radially non-increasing function $u:\RR{d} \to \RR{+}$ can be associated to a radius function $r:[0,+\infty[ \to \RR{+}$ by the relation
\begin{equation} \label{link} \{ u\geqslant t \} = B(0,r(t)). \end{equation}
As a consequence, for all $u$, $r$ is non-increasing and conversely any non-increasing function $r$ can be canonically associated to a radially non-increasing function $u$ defined by $u(x)=sup{t>0 \ | \ r(t)\geqslant  |x|}$. Moreover $u\in L^1 \Leftrightarrow r\in L^d$ and supp $u \subset \OO$ $\Leftrightarrow r(0)<R_0 $.

\begin{prop}
For any pair of functions $(u,r)$ satisfying relation (\ref{link}) for almost every $t>0$:
\[ u\in L^1_{c+} (\OO) \Leftrightarrow r\in \mathcal{A} . \] 
\end{prop}

We then write the expression of $\FF$ using the radius functions of $u$ and $u_0$:
\[ \WW(u,\OO) = \omega_d \int_0^{\sup u} r(t)^{d-1-(d-1)} \ud t = \omega_d \sup u = \omega_d |\textup{supp } r | \]
and
\[ ||u-u_0||_{L^1} = \alpha_d \int_0^{+\infty} |r(t)^d-r_0(t)^d| \ud t . \]
So
\[ \FF(u) = F(r) = \omega_d |\text{supp } r | + \frac{\alpha_d^2}{2\tau}
\left( \int_0^{+\infty} |r(t)^d-r_0(t)^d| \ud t  \right)^2 . \]
We deduce the following proposition

\begin{propo}
For any pair of functions $(u,r)$ satisfying relation (\ref{link}), the following properties are equivalent:
\begin{itemize}
\item $u$ is a minimizer of $\FF$ among radially non-increasing functions of $L^1_{c+} (\OO)$ 
\item $r$ is a minimizer of $F$ in $\mathcal{A}$.
\end{itemize}
\end{propo}

Now we will see that a minimizer is in fact a truncation of the initial datum $r_0$. Let $r$ be any candidate, consider $\tilde{r}(t) = r_0(t)\ind{t<\text{supp }r}$ and compute $F(\tilde{r})$. As supp $\tilde{r} \subset$ supp $r$ and $\tilde{r}=r_0$ on supp $r$, we have
\[ F(\tilde{r}) \leqslant F(r) .\]
Thus, we look for a minimizer such that
\[ r(t)=r_0(t)\ind{t<\lambda} . \]
Our non-trivial minimization problem on $L^1_{c+} (\OO)$ is now reduced to a one parameter ($\lambda$) minimization problem. Note $a=|$ supp $r_0 |$ and observe that we can only consider $0\leqslant \lambda \leqslant a$. \\
We study the minimization of
\[ \begin{array}{rrccl}
f & : & [0,a] & \to & \RR{+} \\
  &   &  \lambda  & \mapsto & \displaystyle \omega_d\lambda+\frac{\alpha_d^2}{2\tau}
\left( \int_\lambda^a r_0(s)^d \ud s \right)^2 .
\end{array} \]
$r_0^d \in L^1$ so $f$ is continuous (then admits a minimum) and almost everywhere differentiable:
\[ f'(\lambda) = \omega_d-\frac{\alpha_d^2}{\tau}r_0(\lambda)^d\int_\lambda^a r_0(s)^d \ud s .\]
$r_0$ is non-increasing and non-negative and, if it is constant on an interval then the integral cannot be constant on the same interval. This is why $f'$ is (strictly) increasing. Note that $f'(a)=\omega_d>0$ and
\[ f'(0)=\omega_d-\frac{\alpha_d^2}{\tau}r_0(0)^d||r_0||_{L^d}<0 \]
for a small enough $\tau$. If $r_0$ is continuous then there exists a unique $\lambda$ such that $f'(\lambda)=0$ otherwise we can define \[ \lambda = \sup (s \ | \ f'(s)\leqslant 0 ) . \]
Thus, there exists a unique $\lambda$ minimizing $f$ and, if $r_0$ is continuous, it is defined by the relation
\begin{equation} \label{lambda}
r_0(\lambda)^d\int_\lambda^a r_0(s)^d \ud s = \frac{\omega_d}{\alpha_d^2}\tau .
\end{equation}
Note that if $\tau$ is not small enough, the minimum $\lambda$ is equal to zero. Moreover for all $z\in [0,a[$ \[r_0(z)^d\int_z^a r_0(s)^d \ud s > 0, \]
so there exists a $\tau>0$ small enough so that $f'(z)<0$ and then $\lambda>z$. Consequently
\[ \lambda \stackrel{\tau \to 0}{\longrightarrow} a . \]

\begin{rmq}
Note that $u_0$ and $r_0$ do not share the same discontinuity points. Indeed, in the example of figure \ref{initial}, the point $B$ is a discontinuity point for $u_0$ (a discontinuity sphere actually) but it is a point where $r_0$ is locally constant. Conversely, $u_0$ is locally constant at point $A$ and it is a discontinuity point for $r_0$.
\end{rmq}

\subsection{Minimizing movement.}

According to the minimizing movement principle, we will construct a minimizing sequence. For each iteration, we can reduce to the minimization of $f$ by updating $a=\lambda_n$. We suppose now that the radius function $r_0$ is continuous except on $a$ (later we will need $r_0(a)>0$ which, by definition of $a$, induces a discontinuity)

Let $\lambda_0=a$, consider the functions
\[ f_n(\lambda) = \omega_d\lambda+\frac{\alpha_d^2}{2\tau}
\left( \int_\lambda^{\lambda_n} r_0(s)^d \ud s \right)^2 \]
and define $\lambda_{n+1}$ as the (unique) minimizer of $f_n$ on $[0,\lambda_n]$. The sequence $(\lambda_n)$ is non-increasing. The value of $\tau$ is fixed, as long as it is small enough, $\lambda_n$ is decreasing and satisfies (\ref{lambda}) i.e.
\[ r_0(\lambda_{n+1})^d\int_{\lambda_{n+1}}^{\lambda_n}
r_0(s)^d \ud s = \frac{\omega_d}{\alpha_d^2}\tau . \]
$(\lambda_n)$ is decreasing, and then, as we saw in Section \ref{forme min}, $\lambda_n$ may possibly be small enough with respect to $\tau$ to ensure $\lambda_{n+1}=0$. If that case occurs then for all $i>n$, $\lambda_i=0$ i.e. the minimizing sequence goes to zero in a finite number of iterations and then remains null.

As $r_0$ is non-increasing, we have
\[ r_0(\lambda_{n+1})^d r_0(\lambda_n)^d (\lambda_n-\lambda_{n+1})
\leqslant \frac{\omega_d\tau}{\alpha_d^2} \leqslant
r_0(\lambda_{n+1})^{2d}(\lambda_n-\lambda_{n+1}) . \]
In order to have a global estimation on $\lambda_n$, we need to suppose $r_0(a)>0$ which yields the aforementioned discontinuity. We obtain the double inequality
\begin{equation} \label{encadrement}
\frac{\omega_d\tau}{r_0(0)^{2d}\alpha_d^2}
\leqslant \lambda_n-\lambda_{n+1} \leqslant
\frac{\omega_d\tau}{r_0(a)^{2d}\alpha_d^2} .
\end{equation}

\begin{rmq}
The first inequality above shows that, beyond some threshold, we have $\lambda_n=0$. Indeed, by summing we deduce that
\[ \lambda_n \leqslant a- {\omega_d}{r_0(0)^{2d}\alpha_d^2} n\tau . \]
As we necessarily have $\lambda_n\geqslant 0$, that means
\[ n\tau \leqslant \frac{ar_0(0)^{2d}\alpha_d^2}{\omega_d} . \]
That gives us a bound, depending on $\tau$, of the threshold when $\lambda_n$ vanishes.
\end{rmq}

The second inequality allows us to prove the minimizing movement convergence. Define
\[ \begin{array}{rrccl}
\lambda_\tau & : & \RR{+} & \to & [0,a] \\
  &   &  t  & \mapsto & \displaystyle \lambda_n+\frac{t-n\tau}{\tau}(\lambda_{n+1}-\lambda_n)
\quad \text{for } t\in[n\tau,(n+1)\tau] .
\end{array} \]
$\lambda_\tau$ is a family of continuous and piecewise affine functions satisfying, using (\ref{encadrement}) 
\[ |\lambda_\tau'(t)| \leqslant \frac{\omega_d}{r_0(a)^{2d}\alpha_d^2} . \]
Using the previous remark, we know that $\lambda_\tau (t) = 0$ for $t\geqslant \frac{ar_0(0)^{2d}\alpha_d^2}{\omega_d}$ and then, using Ascoli theorem, we obtain the following result.

\begin{propo}
There exists a limit function $\lambda : \RR{+} \to [0,a]$ such that $\lambda_\tau \to \lambda$ uniformly when $\tau \to 0$, up to a subsequence. $\lambda$ is Lipschitz continuous and non-increasing.
\end{propo}

Thus, the minimizing movement exists, is continuous and consists of an erosion of the level of the initial datum.

\begin{rmq}
In standard minimizing movement theory the discrete flow is interpolated using $\lambda_\tau(t) = \lambda_{[t/\tau]}$ i.e a sequence of piecewise constant functions. For convenience, we rather used continuous piecewise affine functions. As
\[ |\lambda_{[t/\tau]}-\lambda_\tau(t)| \leqslant \frac{\omega_d}{r_0(a)^{2d}\alpha_d^2} \tau \]
both definitions of $\lambda_\tau$ have the same limit if at least one exists.
\end{rmq}

Furthermore, we can have more information on the limit function $\lambda$.
\begin{propo}
Suppose $r_0$ is continuous. At every point $t$ such that $\lambda(t) > 0$, $\lambda$ is differentiable and satisfies
\begin{equation} \label{loi}
\lambda'(t) = -\frac{\omega_d}{\alpha_d^2} \frac{1}{r_0(\lambda(t))^{2d}} . \end{equation}
\end{propo}

\begin{proof}
Let $\ee>0$ be such that $\lambda(t+\ee)>0$. For all $\tau>0$, write
\[ t=p\tau+\eta \qquad \text{and} \qquad t+\ee=(p+q)\tau + \eta' \]
with $p,q\in \NN$ and $0\leqslant \eta,\eta'<\tau$. Then
\[ \lambda_p \to \lambda(t) \qquad \text{and} \qquad \lambda_{p+q} \to \lambda(t+\ee) \]
when $\tau \to 0$. \\
For all integers between $p$ and $p+q$, use the double inequality (\ref{encadrement}) with $\lambda_p$ and $\lambda_{p+q}$ instead of $0$ and $a$
\[ \frac{\omega_d\tau}{r_0(\lambda_{p+q})^{2d}\alpha_d^2}
\leqslant \lambda_n-\lambda_{n+1} \leqslant
\frac{\omega_d\tau}{r_0(\lambda_p)^{2d}\alpha_d^2} .  \]
Summing on $n$ from $p$ to $p+q-1$, we obtain
\[ \frac{\omega_d}{r_0(\lambda_{p+q})^{2d}\alpha_d^2} q\tau
\leqslant \lambda_p-\lambda_{p+q} \leqslant
\frac{\omega_d}{r_0(\lambda_p)^{2d}\alpha_d^2} q\tau .  \]
Writing $q\tau$ in term of $\ee$, $\eta$ and $\eta'$ and dividing by $\ee$, one can see that
\[ \frac{\omega_d}{r_0(\lambda_{p+q})^{2d}\alpha_d^2} \frac{\ee+\eta-\eta'}{\ee}
\leqslant \frac{\lambda_p-\lambda_{p+q}}{\ee} \leqslant
\frac{\omega_d}{r_0(\lambda_p)^{2d}\alpha_d^2} \frac{\ee+\eta-\eta'}{\ee} .  \]
Passing to the limit $\tau \to 0$ we have
\[ \frac{\omega_d}{r_0(\lambda(t+\ee))^{2d}\alpha_d^2}
\leqslant -\frac{\lambda(t+\ee)-\lambda(t)}{\ee} \leqslant
\frac{\omega_d}{r_0(\lambda(t))^{2d}\alpha_d^2} .  \]
Passing to the limit $\ee \to 0$ yields the right derivative of $\lambda$ on $t$. Proceeding similarly (taking care of the sign), we obtain
\[ \frac{\omega_d}{r_0(\lambda(t))^{2d}\alpha_d^2}
\leqslant -\frac{\lambda(t-\ee)-\lambda(t)}{-\ee} \leqslant
\frac{\omega_d}{r_0(\lambda(t-\ee))^{2d}\alpha_d^2} .  \]
Then, passing to the limit $\ee \to 0$, we obtain the left derivative.
\end{proof}

Hence, we now understand the minimizing movement for a radially non-increasing function whose highest non empty superlevel set is not reduced to a point. The ordinary differential equation (\ref{loi}) contains a difficulty if $r_0(a)=0$ because, in that case, the initial datum is truncated with an infinite initial speed. After the starting point, the movement is clear. Indeed, on every interval of the form $[T,+\infty[$ with $T>0$, the minimizing movement exists (we have the double inequality (\ref{encadrement}) except for the first times $\lambda_0, \lambda_1,\dots$ and so the uniform convergence remains) and follows the law (\ref{loi}).

We assumed the continuity of $r_0$ on $[0,a[$. Nevertheless, (\ref{loi}) remains valid even if $r_0$ is not continuous but only non-increasing. At the points where $r_0$ is continuous, we have ($\ref{loi}$). At any other points, we do not know if $\lambda$ is differentiable but the growth ratio is controlled. In the sequel we will write $r_0(s+)$ and $r_0(s-)$ for the right and left limits at any point $s$. Note that $r_0$ is non-increasing and therefore left and right limits exist (observe that $r_0$ is left continuous).

\begin{propo}
At every point $t$ where $\lambda(t) > 0$, we have the following control of the growth ratio
\[ -\frac{\omega_d}{r_0(\lambda(t)+)^{2d}\alpha_d^2}
\leqslant \liminf_{\ee\to 0} \frac{\lambda(t+\ee)-\lambda(t)}{\ee}
\leqslant \limsup_{\ee\to 0} \frac{\lambda(t+\ee)-\lambda(t)}{\ee}
\leqslant -\frac{\omega_d}{r_0(\lambda(t)-)^{2d}\alpha_d^2} \]
for both $\ee \to 0^+$ and $\ee \to 0^-$.
\end{propo}

\begin{proof}
May $r_0$ be continuous or not, we still have the existence of the sequence $(\lambda_n)$ and, by definition,
\[ \lambda_{n+1} = \sup (s \ | \ f_n'(s)\leqslant 0 ) . \]
Noticing $(\lambda_n)$ is decreasing (for non null terms) and $f_n'$ increasing we have for (almost) every $s\in]\lambda_{n+1},\lambda_n[$, $f_n'(s) \geqslant 0$ thus
\[ r_0(s)^d \int_s^{\lambda_n}r_0(u)^d \ud u
\leqslant \frac{\tau\omega_d}{\alpha_d^2} . \]
So
\[ r_0(\lambda_n)^{2d}(\lambda_n-s) \leqslant \frac{\tau\omega_d}{\alpha_d^2} . \]
Now take $s\to \lambda_{n+1}$ to have
\[  \lambda_n-\lambda_{n+1} \leqslant \frac{\omega_d\tau}{r_0(\lambda_n)^{2d}\alpha_d^2} . \]
Identically, for (almost) every $s<\lambda_{n+1}$, $f_n'(s) \leqslant 0$ and then
\[ r_0(s)^d \int_s^{\lambda_n}r_0(u)^d \ud u
\geqslant \frac{\tau\omega_d}{\alpha_d^2} . \]
So
\[ r_0(s)^{2d}(\lambda_n-s) \geqslant \frac{\tau\omega_d}{\alpha_d^2} . \]
Letting $s\to \lambda_{n+1}$, one gets
\[  \lambda_n-\lambda_{n+1} \geqslant \frac{\omega_d\tau}{r_0(\lambda_{n+1}-)^{2d}\alpha_d^2} . \]
For global estimation, we have thus exactly double inequality(\ref{encadrement}) and then the existence of the minimizing movement.

Similarly, to satisfy ($\ref{loi}$), the same technique works up to the following small modifications as follow. Notice that $r_0(\lambda_{p+q}-)\leqslant r_0(\lambda_{p+q+1})$ and $\lambda_{p+q+1} \to \lambda(t+\ee)$ because
$|\lambda_{p+q}-\lambda_{p+q+1}| \leqslant \frac{\omega_t}{r_0(a)^{2d}\alpha_d^2}\tau$. Therefore
\begin{equation} \label{coince} \frac{\omega_d}{r_0(\lambda_{p+q+1})^{2d}\alpha_d^2} \frac{\ee+\eta-\eta'}{\ee}
\leqslant \frac{\lambda_p-\lambda_{p+q}}{\ee} \leqslant
\frac{\omega_d}{r_0(\lambda_p)^{2d}\alpha_d^2} \frac{\ee+\eta-\eta'}{\ee} .  \end{equation}
The values of the sequence $(\lambda_n)$ depends on $\tau$ therefore, for $m=[s/\tau]$ it is not clear how $\lambda_m \to \lambda(s)$ when $\tau$ vanishes hence, we use the inequalities
\[ \frac{1}{r_0(\lambda(s)-)} \leqslant
\liminf_{\tau \to 0} \frac{1}{r_0(\lambda_m)} 
\leqslant \limsup_{\tau \to 0} \frac{1}{r_0(\lambda_m)}
\leqslant \frac{1}{r_0(\lambda(s)+)} .\]
Passing to lim inf in the left hand side of double inequality (\ref{coince}) (with $s=t+\ee$ and $m=p+q+1$) and passing to the lim sup in the right hand side (with $s=t$ and $m=p$) we have
\[ \frac{\omega_d}{r_0(\lambda(t+\ee)-)^{2d}\alpha_d^2}
\leqslant \frac{\lambda(t)-\lambda(t+\ee}{\ee} \leqslant
\frac{\omega_d}{r_0(\lambda(t)+)^{2d}\alpha_d^2} .  \]
Thanks to monotonicity ($t\mapsto \lambda(t)$ is non-increasing) we have
\[ \frac{1}{r_0(\lambda(t+2\ee))} \leqslant \frac{1}{r_0(\lambda(t+\ee)-)}
\quad \text{and} \quad
\frac{1}{r_0(\lambda(t)+)} \leqslant \frac{1}{r_0(\lambda(t-\ee))} . \]
So
\[ \frac{\omega_d}{r_0(\lambda(t+2\ee))^{2d}\alpha_d^2}
\leqslant \frac{\lambda(t)-\lambda(t+\ee)}{\ee} \leqslant
\frac{\omega_d}{r_0(\lambda(t-\ee))^{2d}\alpha_d^2} .  \]
As $t+2\ee$ decreases to $t$ when $\ee \to 0^+$ then $\lambda(t+2\ee)$ increases to $\lambda(t)$ and so $r_0(\lambda(t+2\ee)) \to r_0(\lambda(t)-)$ when $\ee \to 0^+$. Similarly $r_0(\lambda(t-\ee)) \to r_0(\lambda(t)+)$ when $\ee \to 0^+$. Thus we can pass to lim inf and lim sup in the previous inequalities to have
\[ -\frac{\omega_d}{r_0(\lambda(t)+)^{2d}\alpha_d^2}
\leqslant \liminf_{\ee\to 0^+} \frac{\lambda(t+\ee)-\lambda(t)}{\ee}
\leqslant \limsup_{\ee\to 0^+} \frac{\lambda(t+\ee)-\lambda(t)}{\ee}
\leqslant -\frac{\omega_d}{r_0(\lambda(t)-)^{2d}\alpha_d^2} . \]
Using exactly the same arguments for the left derivative we obtain
\[ -\frac{\omega_d}{r_0(\lambda(t)+)^{2d}\alpha_d^2}
\leqslant \liminf_{\ee\to 0^+} \frac{\lambda(t-\ee)-\lambda(t)}{-\ee}
\leqslant \limsup_{\ee\to 0^+} \frac{\lambda(t-\ee)-\lambda(t)}{-\ee}
\leqslant -\frac{\omega_d}{r_0(\lambda(t)-)^{2d}\alpha_d^2} . \]
\end{proof}

\section{Conclusion.}

Let us recall our objective: we were investigating the best way to decrease the $L^1$ relaxation of the scale invariant Willmore energy
\[ W(u) = \frac{1}{(d-1)^{d-1}}\mathlarger\int_{\RR{d}} \left| \Div \left( \frac{\nabla u}{|\nabla u|} \right) \right|^{d-1}
|\nabla u| \ud x . \]
We considered a radially non-increasing initial function $u_0$ and its radius function $r_0$ as in figure \ref{initial}.

\begin{figure}[!ht]

\begin{center}
\begin{tikzpicture}[scale=0.6]
\draw (0,0) node[below left] {$0$} ;
\draw [->] (-1,0) -- (10,0) node[below right] {$|x|$} ;
\draw [->] (0,-1) -- (0,6) node[left] {$u_0$} ;
\draw [thick] (0,5) -- (1,5) 
plot [domain=1:4,samples=100] (\x,{ 5-(2/9)*(\x-1)^2 })
plot [domain=4:6,samples=100] (\x,{ 1-(1/8)*(\x-6)^3 }) 
(6,1) -- (7,1)
plot [domain=7:9,samples=100] (\x,{ 1-(1/8)*(\x-7)^3 }) ;
\draw [dashed,thick] (4,3) -- (4,2);
\draw (-0.2,4) -- (3.5, 4)  (-0.5,4) node {$\lambda$} ;
\draw (6.5,1.3) node {$A$} ;
\draw (4.3,2.5) node {$B$} ;
\draw [dashed] (3.12,4.2) -- (3.12,-0.1) ;
\draw [<->] (0.1,-0.2) -- (3.1,-0.2) ;
\draw (1.6,-0.8) node {$r_0(\lambda)$} ;
\end{tikzpicture}
\end{center}

\caption{Initial condition $u_0$.}
\label{initial}

\end{figure}
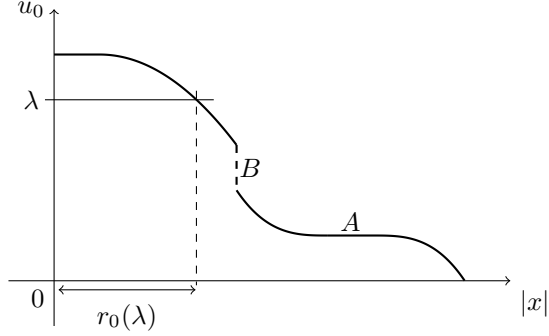

For $\OO=B(0,R_0)$, we proved that the minimizing movement in the space
\[ L^1_{c+} (\OO) = \left\{ u\in L^1(\RR{d}) \ \big| \ u\geqslant 0,
\text{ supp}(u)\subset \OO  \right\} \]
is given by a truncation of the initial function. Collecting our above results, we proved the following.

\begin{thm} \label{mvt min rad}
Let $u_0\in L^1_{c+} (\OO)$ be a bounded and radially non-increasing initial condition and $r_0$ its radius function. If the highest non empty superlevel set of $u_0$ is not reduced to $\{0\}$ (meaning $r_0(\sup u_0)>0$) then there exists a minimizing movement for $\WW$ starting from $u_0$. \\
Moreover, there exists $\lambda:[0,+\infty[\to \RR{+}$ such that the minimizing movement is given by the map
\[ \begin{array}{rcl}
\RR{+} & \to & L^1_{c+} (\OO) \\
t  & \mapsto & u_t \end{array} \]
with
\[ u_t(x)=\min(u_0(x),\lambda(t)). \]
Finally the function $\lambda$ is Lipschitz continuous, non-increasing, satisfies $\lambda(0)=\sup u_0$ and vanishes in finite time.
\end{thm}

Moreover, we obtained some interesting information on $\lambda$ depending on the continuity of $r_0$.

\begin{thm} \label{EDO}
With the notations of Theorem \ref{mvt min rad}, for all $t\in\RR{+}$ such that $\lambda(t)>0$.
If $\lambda(t)$ is a continuity point of $r_0$ then $\lambda$ is differentiable on $t$ and
\[ \lambda'(t) = -\frac{\omega_d}{\alpha_d^2} \frac{1}{r_0(\lambda(t))^{2d}} \]
Otherwise, we have the following control on the growth ratio
\[ -\frac{\omega_d}{r_0(\lambda(t)+)^{2d}\alpha_d^2}
\leqslant \liminf_{\ee\to 0} \frac{\lambda(t+\ee)-\lambda(t)}{\ee}
\leqslant \limsup_{\ee\to 0} \frac{\lambda(t+\ee)-\lambda(t)}{\ee}
\leqslant -\frac{\omega_d}{r_0(\lambda(t)-)^{2d}\alpha_d^2} \]
for both $\ee \to 0^+$ and $\ee \to 0^-$.
\end{thm}

\noindent An illustration of the minimizing movement is shown in figure \ref{evolution}.

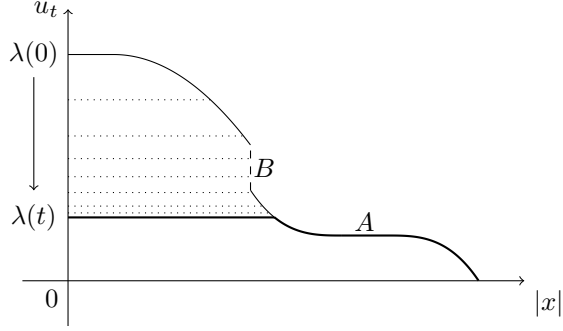
\begin{figure}[!ht]

\begin{center}
\begin{tikzpicture}[scale=0.6]
\draw (0,0) node[below left] {$0$} ;
\draw [->] (-1,0) -- (10,0) node[below right] {$|x|$} ;
\draw [->] (0,-1) -- (0,6) node[left] {$u_t$} ;
\draw (0,5) -- (1,5) 
plot [domain=1:4,samples=100] (\x,{ 5-(2/9)*(\x-1)^2 })
plot [domain=4:4.52,samples=100] (\x,{ 1-(1/8)*(\x-6)^3 }) ;
\draw [thick] (0,1.4) -- (4.52,1.4)
plot [domain=4.52:6,samples=100] (\x,{ 1-(1/8)*(\x-6)^3 }) 
(6,1) -- (7,1)
plot [domain=7:9,samples=100] (\x,{ 1-(1/8)*(\x-7)^3 }) ;
\draw [dotted] (0,4) -- (3.12,4) ;
\draw [dotted] (0,3.2) -- (3.84,3.2) ;
\draw [dotted] (0,2.7) -- (4,2.7) ;
\draw [dotted] (0,2.3) -- (4,2.3) ;
\draw [dotted] (0,1.95) -- (4.03,1.95) ;
\draw [dotted] (0,1.65) -- (4.26,1.65) ;
\draw [dotted] (0,1.5) -- (4.4,1.5) ;
\draw [dashed] (4,2) -- (4,3) ;
\draw (6.5,1.3) node {$A$} ;
\draw (4.3,2.5) node {$B$} ;
\draw (-0.75,5) node {$\lambda(0)$} ;
\draw [->] (-0.75,4.5) -- (-0.75,2) ;
\draw (-0.75,1.4) node {$\lambda(t)$} ;
\end{tikzpicture}
\end{center}

\caption{Minimizing movement $u_t$.}
\label{evolution}

\end{figure}

In zone $B$, the erosion speed is locally constant and in zone $A$, it possibly jumps.

\section*{Perspectives.}

We were not able to see what happens if $r_0(\sup u_0)=0$. Indeed, the minimizing sequence convergence is given by the Ascoli theorem and we need a uniform control on the speed decrease of $(\lambda_n)$. Looking at the ODE (Theorem \ref{EDO}), the initial speed is non-finite if $r_0(\sup u_0)=0$ therefore we cannot hope to have such a control. For now, we do not know if, in that case, the flow consists on an erosion with a non-finite initial speed. Otherwise, the flow may consist on cutting off a small height and then the erosion would begin following the ODE.

More generally, we can wonder what is the flow of the characteristic function of any set $E$. Is it an erosion? Or do the superlevel sets evolve? Maybe both? We proved in this paper that for a ball $E=B(0,R)$, the gradient flow starting from $u_0=\ind{B(0,R)}$ is an erosion given by
\[ u_t = (1-c_Rt)\ind{B(0,R)} \]
where $c_R=\frac{\omega_d}{\alpha_d^2R^{2d}}$. However, balls have minimal scale invariant Willmore energy which can explain that the superlevel sets do not evolve. The only way to reduce their energies is then to be smaller and smaller. If the set $E$ is not minimal then we can have both behaviours: decrease of the height and/or deformation of the superlevel sets. Perhaps if the Willmore energy is small enough then it is better to be rigid on superlevel sets and to diminish the height.

\appendix

\section{Appendix: Proof of the lemmas.}

We prove the technical lemmas.

\renewcommand{\thethmref}{\ref{inegalite}}

\begin{thmref}[coarea inequality]
Let $u\in L^1(\OO)$, we have
\[ \int_{-\infty}^{+\infty} \Wp(\{u\geqslant t\},\OO) \ud t \quad \leqslant \quad \Wp(u,\OO) .\]
\end{thmref}

\begin{proof}
It is obvious if $\Wp(u,\OO) = +\infty$. If not, let $(u_n)$ be a sequence of functions of class $C^2$ with compact support in $\OO$ such that
\[ u_n \longrightarrow  u \quad \text{in } L^1(\OO) \qquad \text{and} \qquad W_p(u_n) \longrightarrow  \Wp(u,\OO) .\]
Thanks to lemma \ref{ev-ga}, up to a subsequence, we have that for almost every $t$ in $\RR{}$
\[ \{ u_n\geqslant t\} \longrightarrow  \{ u \geqslant t\} \quad \text{in } L^1(\OO) . \]
Moreover, by Sard's Lemma, for almost every $t$, $\{ u_n\geqslant t\}$ is a regular set of $\RR{d}$.
By definition of $\Wp$, we have
\[ \Wp(\{u\geqslant t\},\OO) \ \leqslant \ \liminf_{n\to +\infty} W_p(\{ u_n \geqslant t\}) . \]
Integrating with respect to $t$ and using the Fatou's lemma and the coarea formula (Proposition \ref{coarea regulier}), we get
\[ \begin{array}{lll}
\displaystyle \int_{-\infty}^{+\infty}  \Wp(\{u\geqslant t\},\OO) \ud t & \leqslant &
\liminf \displaystyle \int_{-\infty}^{+\infty} W_p(\{u_n\geqslant t\}) \ud t \\
 & = & \liminf W_p(u_n,\OO) \\
 & = & \Wp(u,\OO) . \end{array} \]
\end{proof}

\renewcommand{\thethmref}{\ref{ev-ga}}

\begin{thmref}[$L^1$ norm]
For measurable functions $u,v : \RR{d} \mapsto \RR{}$ we have
\[ \int_{\RR{d}} |u-v|\ud x =
\int_{-\infty}^{+\infty} |\{u\geqslant t \} \Delta \{ v\geqslant t\}| \ud t .\]
\end{thmref}

\begin{proof}
Pick $t\in \RR{}$ and $x \in \RR{d}$. We have
\[ |\ind{\{u\geqslant t\}}(x) - \ind{\{v\geqslant t\}}(x)| = 
\left\{ \begin{array}{ll}
1 & \qquad \text{if min($u(x)$,$v(x)$) $<t<$ max($u(x)$,$v(x)$) } \\
0 & \qquad \text{else} . \end{array} \right. \]
Then
\[ \int_{-\infty}^{+\infty} |\ind{\{u\geqslant t\}}(x) - \ind{\{v\geqslant t\}}(x)| \ud t
= \max(u(x),v(x)) - \min(u(x),v(x)) = |u(x)-v(x)| . \]
Integrating on $x\in \RR{d}$, using the Fubini theorem and noticing that
\[ \int_{\RR{d}} |\ind{\{u\geqslant t\}}(x) - \ind{\{v\geqslant t\}}(x)| \ud x =
|\{u\geqslant t \} \Delta \{ v\geqslant t\}|, \]
the result is proved.
\end{proof}

\renewcommand{\thethmref}{\ref{recollement}}

\begin{thmref}[reattachment]
Let $a,b,\alpha,\beta \in \RR{}$ such that $a,b<0$. There exists a decreasing function $f\in C^2([0,1])$ such that
\[ \begin{array}{cc}
f(0)=1 & f(1)=0 \\ f'(0)=a & f'(1)=b \\ f''(0)=\alpha & f''(1)=\beta . \end{array} \]
\end{thmref}

\begin{proof}
\textbf{Step 1:} first we construct a $C^2$ non increasing function $f_0$ satisfying
\[ \begin{array}{cc}
f_0(0)=1 & f_0(1)>0 \\ f_0'(0)=a & f_0'(1)=b \\ f_0''(0)=\alpha & f_0''(1)=\beta . \end{array} \]
Let $\ee >0$, we define $f_0'=0$ on $[\ee,1-\ee]$ and we enforce the boundary conditions on $f_0'$ and $f_0''$ with polynomials. Then we set
\[ f_0(t)=1+\int_0^tf_0'(s)\ud s \]
and choose $\ee$ small enough to define $f_0$.

We introduce the following polynomials
\[ \begin{array}{lcl}
P=(\lambda_1+\mu_1(X-\ee))(X-\ee)^2 & \text{ with } & \lambda_1 = 3a/\ee^2 + \alpha /\ee \\
 & & \mu_1 = 2a/\ee^3 + \alpha /\ee^2 \\
Q=(\lambda_2+\mu_2(1-\ee-X))(1-\ee-X)^2 & \text{ with } & \lambda_2 = 3b/\ee^2 - \beta /\ee \\
 & & \mu_2 = 2b/\ee^3 - \beta /\ee^2 \\
\end{array} \]
These two polynomials satisfy the conditions:
\[ \begin{array}{cccc}
P(0)=a & P(\ee)=0 & Q(1-\ee)=0 & Q(1)=b \\ P'(0)=\alpha & P'(\ee)=0 & Q'(1-\ee)=0 & Q'(1)=\beta .
\end{array} \]
Moreover, taking $\ee$ small enough, say $\ee < \min(3|a|/|\alpha|,3|b|/|\beta|)$, then these polynomials are negative where we need them to be.
\[ \begin{array}{cl}
P < 0 & \quad \text{on } [0,\ee[ \\ Q < 0 & \quad \text{on } ]1-\ee,1] . \end{array} \]
An easy calculation shows that
\[ \int_0^\ee P(t)\ud t + \int_{1-\ee}^1Q(t) \ud t = \frac{a+b}{2}\ee + \frac{\alpha-\beta}{12}\ee^2 < 0. \]
For $\ee<1/2$, let us define the function
\[ g(t) = \left\{ \begin{array}{ll}
P(t) & \quad \text{if } t\in [0,\ee] \\
0 & \quad \text{if } t\in [\ee,1-\ee] \\
Q(t) & \quad \text{if } t\in [1-\ee,1] . \end{array} \right. \]
The function $g$ is $C^1$, non positive and satisfies the boundary conditions $g(0)=a$, $g(1)=b$, $g'(0)=\alpha$ and $g'(1)=\beta$. Now we choose $\ee < \min(1/2,3|a|/|\alpha|,3|b|/|\beta|)$ such that
\[ 1 + \frac{a+b}{2}\ee + \frac{\alpha-\beta}{12}\ee^2 > 0 \]
and we define
\[ f_0(t)=1+\int_0^tg(s)\ud s . \]

\textbf{Step 2:} now let $\ff$ be a smooth function on $[0,1]$ such that $\ff$ is positive on $]0,1[$, vanishes at $t=0$ and $t=1$ at both zeroth and first orders and $\int_0^1 \ff(t) \ud t =1$.
Let $A>0$ such that \[ 1 + \frac{a+b}{2}\ee + \frac{\alpha-\beta}{12}\ee^2 - A =0 . \]
Take \[ f(t) = f_0(t) - A \int_0^t \ff(s)\ud s . \]
This function verifies the thesis of the lemma.
\end{proof} 

\renewcommand{\thethmref}{\ref{regulier}}

\begin{thmref}[Relaxation on regular sets]
Let $p>1$ and let $E$ be a $C^2$ compact subset included in $\OO$. Then
\[ \Wp(E,\OO) = W_p(E) .\]
\end{thmref}

\begin{proof}
In the literature, we can found proofs with the energy $P(E)+W_p(E)$ where $P$ stands for the $BV$ perimeter. In our context, the proof is exactly the same since we have a control on the perimeter by $\Wp$. Indeed, $\od E$ is a $C^2$ compact hypersurface, thus denoting $V$ the vector field defined by $V(x)=x$ and computing $\textup{div}_{\od E} V = d-1$ we have, by the divergence theorem,
\[ \int_{\od E}\textup{div}_{\od E}V \ud \HH{d-1} = - \int_{\od E}V\cdot \overrightarrow{H} \ud \HH{d-1} . \]
Then, using Hölder inequality with $1/p+1/q=1$, 
\[ (d-1)P(E) = \int_{\od E}\textup{div}_{\od E}V \ud \HH{d-1} \leqslant
	\left( \int_{\od E}|V|^q \ud \HH{d-1} \right)^{1/q}
	\left( \int_{\od E}|\overrightarrow{H}|^p \ud \HH{d-1} \right)^{1/p} . \]
Let $M$ be a positive real such that $\OO \subset B(0,M)$, we obtain
\[ (d-1)P(E) \leqslant M W_p(E)^{1/p} P(E)^{1/q} . \]
Thus
\[ P(E) \leqslant \left( \frac{M}{d-1} \right)^p W_p(E) . \]

We only provide the main line of the proof and we refer to \cite[Theorem 4]{Ambrosio2003}. Let $E_n$ be a sequence of $C^2$ sets included in $\OO$ such that $E_n \to E$ in $L^1(\OO)$. Without loss of generality, we suppose
$\displaystyle \sup_n W_p(E_n) < +\infty$ and then, by the previous inequality, $\displaystyle \sup_n P(E_n) < +\infty$. Consider the associated unit varifolds $V_n = v(\od E_n,1)$. As $\od E_n$ is regular, the classical mean curvature and the generalized mean curvature coincide and are denoted by $H_n$. By previous bounds on $W_p$ and $P$, the total mass and the first variation of $V_n$ are uniformly bounded. Then, by Allard's compactness Theorem, up to a subsequence still denoted by $n$, $V_n$ converges weakly-$\star$ to some integral varifold $V = v(\Sigma,\theta)$. We denote by $\mu_n$ and $\mu$ the masses of $V_n$ and $V$ respectively. Let $B(x,r)$ be a ball in $\OO$, by lower semicontinuity of the $BV$ perimeter, we have
\[ P(E,B(x,r)) \leqslant \liminf_{n\to \infty} P(E_n,B(x,r)) = \liminf_{n\to \infty} \mu_n(B(x,r)) = \mu(B(x,r)) .\]
Therefore $\od E \subset \textup{supp} V$ i.e. $\theta(x) \geqslant 1$ for $\HH{d-1}$-almost every $x \in \od E$.
By lower semicontinuity of the Willmore energy under varifolds convergence following from \cite[example 2.36]{Ambrosio2000}) and \cite{Ambrosio2003},  $V$ admits a generalized mean curvature denoted by $H$. Moreover, by the locality of the mean curvature (see \cite{Ambrosio2003,Leonardi2009} and the extension for $p>1$ and any dimension in \cite{Menne2013}), $H(x)$ coincides with the classical mean curvature of $\od E$ for $\HH{d-1}$-almost every $x \in \od E$. Therefore we have
\[ W_p(E) = \int_{\od E} |H|^p \ud \HH{d-1} \leqslant \int_{\Sigma} \theta|H|^p \ud \HH{d-1}
	\leqslant \liminf_{n \to +\infty} \int_{\od E_n} |H_n|^p \ud \HH{d-1} = \liminf_{n \to +\infty} W_p(E_n) .\]
Thus $W_p(E) \leqslant \Wp(E,\OO)$.
Taking the constant sequence $E_n = E$ we have $\Wp(E,\OO) \leqslant W_p(E)$, hence the lemma is proved. 
\end{proof}

\renewcommand{\thethmref}{\ref{inclusion}}

\begin{thmref}
Let $\lambda, \mu>0$. The function $(A,B) \mapsto |A\Delta B|$, defined on the collection of pairs $(A,B)$ such that $|A|=\lambda$ and $|B|=\mu$, reaches its minimum whenever $A\subset B$ or $B\subset A$ (up to a negligible set) and its minimum is $|A\Delta B| = |\lambda-\mu|$.
\end{thmref}

\begin{proof}
\[ |A\Delta B| = |A\smallsetminus B| + |B\smallsetminus A|
= |A|+|B|-2|A\cap B| = \lambda + \mu - 2|A\cap B| . \]
We have $|A\cap B| \leqslant \min(|A|,|B|)$ where equality found when one set is included into the other (up to a negligible set). Then $(A,B) \mapsto |A\Delta B|$ achieves its minimum for $A\subset B$ or $B\subset A$. \\
If $A\subset B$ then $|A\cap B|=|A|$ and $\mu \geqslant \lambda$ thus $|A\Delta B|=\mu-\lambda $. \\
If $B\subset A$ then $|A\cap B|=|B|$ and $\mu \leqslant \lambda$ thus $|A\Delta B|=\lambda-\mu $.
\end{proof}

\renewcommand{\thethmref}{\ref{minwillmore}}

\begin{thmref}
Let $E$ be a compact and non negligible subset of $\RR{d}$. Then
\[ \WW(E,\OO) \geqslant \omega_d . \]
\end{thmref}

\begin{proof}
It is a direct consequence of well known results for the scale invariant Willmore energy (see \cite{Jiazu2011} and references therein for instance): for all regular compact hypersurfaces $M$ of $\RR{d}$, we have
\[ \int_M |H|^{d-1} \ud \HH{d-1} \geqslant \omega_d \]
with equality only for spheres. \\
$\WW(E,\OO)<+\infty$ implies there exists a sequence $(E_n)$ of compact subsets of $\OO$ of class $C^2$ such that $E_n \to E$ in $L^1$ and $W(E_n) \to \WW(E,\OO)$ when $n\to +\infty$. Using the result on the boundary $\od E_n$, we have
\[ W(E_n) \geqslant \omega_d .\]
Passing to the limit $n\to +\infty$, we have
\[ \WW(E,\OO) \geqslant \omega_d .\]
\end{proof}

\end{document}